\documentclass[11pt,a4paper]{amsart}

\usepackage[marginparwidth=0pt,margin=20truemm]{geometry} % margin

\usepackage{mypackage} % My packages.
\usepackage{mycommand} % My commands.

\usepackage{comment}

\DeclareMathOperator{\Resultant}{Res}
\renewcommand{\Res}{\Resultant}
\DeclareMathOperator{\modulo}{mod}
\renewcommand{\mod}{\modulo}
\DeclareMathOperator{\LT}{LT}
\DeclareMathOperator{\Disc}{Disc}
\newcommand{\cyc}{\Phi^{\mathrm{cyc}}} 
\newcommand{\pe}{\mathfrak{p}}
\usepackage[pdfencoding=auto,hypertexnames=false]{hyperref}
\hypersetup{colorlinks=false}

\numberwithin{equation}{section} % Setting of equation numbers 
\usepackage{mytheoremeng} % My theorem environments (English) with cleveref package

\begin{document}
% --------------------------------------------------------------------------

\title[Multiplier polynomials]{Arithmetic properties of multiplier polynomials for certain polynomial maps}

\author[Y. Murakami]{Yuya Murakami}
\address{Faculty of Mathematics, Kyushu University, 744, Motooka, Nishi-ku, Fukuoka, 819-0395, JAPAN}
\email{murakami.yuya.896@m.kyushu-u.ac.jp}

\author[K. Sano]{Kaoru Sano}
\address{NTT Institute for Fundamental Mathematics, NTT Communication Science Laboratories, NTT Corporation, 2-4, Hikaridai, Seika-cho, Soraku-gun, Kyoto 619-0237, Japan}
\email{kaoru.sano@ntt.com}

\author[K. Takehira]{Kohei Takehira}
\address{Graduate School of Science, Tohoku University, 
	6-3, Aoba, Aramaki-aza, Aoba-ku, 
	Sendai, 980-8578, Japan}
\email{kohei.takehira.p5@dc.tohoku.ac.jp}

\date{\today}

% --------------------------------------------------------------------------

\begin{abstract}
    We investigate the arithmetic properties of the multiplier polynomials for certain $1$-parameter families of polynomials.
    In particular, we prove integrality theorems of multiplier polynomials for $z^d+c$, $(z-c)z^d + c$ and $z^{d+1}+cz$.
    As a corollary, we obtain the uniform upper bound on the naive height of parabolic parameters of unicritical polynomials.
    Moreover, we determined the quadratic parabolic parameters for $z^2 + c$.
    We also conditionally list parabolic parameters for $z^2 + c$ of fixed degrees.
\end{abstract}

\maketitle

% --------------------------------------------------------------------------

\tableofcontents

% --------------------------------------------------------------------------

\section{Introduction} \label{sec: intro}

% --------------------------------------------------------------------------
For a polynomial $f$ of degree $d$ over an integral domain $\mathcal{O}$ and a positive integer $k$,
$f^{\circ k}$ denotes the $k$-th iterated composition $f\circ f\circ \cdots \circ f$.
The $n$-th {\it dynatomic polynomial} $\Phi_{n}^\ast(z)$ (or $\Phi_{f,n}^{\ast}(z)$ for clarity) of $f$ is
\[
    \Phi_n^\ast(z) = \prod_{k|n} (f^{\circ k}(z) -z)^{\mu(n/k)},
\]
where $ \mu $ is the M\"{o}bius function
(see \cite[Theorem 4.5]{Silbook00}).

The degree of $\Phi_n^\ast(z)$ is $d_n = \sum_{k|n} d^k\mu(\frac{n}{k})$.
For a root $\alpha \in \overline{\Frac(\mathcal{O})}$ of $\Phi_n^\ast$, the multiplier $\omega_n(\alpha)$ is the derivative $(f^{\circ n})'(\alpha)$.
By the chain rule, it is written as
\begin{equation}\label{eq: multiplier}
    \omega_n(\alpha) = \prod_{i=0}^{n-1} f'(f^{\circ i}(\alpha)),
\end{equation}
where $f^{\circ 0}(z)$ is regarded as $z$.
The $n$-th {\it multiplier polynomial} $\delta_n(x)$ of $f(z) = \sum_{i=0}^d a_i z^i \in \mathcal{O}[z]$ is the monic polynomial in $x$ satisfying
\begin{equation}\label{eq: multiplier_polynomial}
    \delta_n(x)^n = \prod_{\Phi_n^\ast(\alpha)=0} (x-\omega_n(\alpha)),
\end{equation}
where $\delta_n$ is indeed an element of $\mathbb{Z}[a_i: 0\leq i \leq d][x]$ (cf \cite{Vivaldi-Hatjispyros}).
Let $\cyc_i(x)$ be the $i$-th cyclotomic polynomial.
For positive integers $m,n$ such that $m \mid n$ and $m<n$, the quantity $\Delta_{n,m}$ is defined by
\begin{equation}\label{eq: Delta}
    \Delta_{n,m} = \Res_x(\cyc_{n/m}(x), \delta_m(x)).
\end{equation}

In addition, we define the quantity $ \Delta_{n,n} $ by
\begin{equation}\label{eq: Delta_nn}
\delta_n(1) = \Delta_{n,n} \prod_{m \mid n, m\neq n} \Delta_{n,m}.
\end{equation}

In \cite[Section 3, Table 1]{Morton-Vivaldi}, the quantity $\Delta_{n,m}$ for $f(z) = z^2+c$ and $n,m \leq 6$ is listed.
Following the list, it is conjectured that $\Delta_{n,m}$ for $f(z) = z^2+c$ is a monic integral polynomial in $4c$ in \cite[Excercise 4.12 (c)**]{Silbook00}.
A generalization of this conjecture for unicritical polynomial maps is proved in Huguin's thesis \cite{Hug21}.
Our first main theorem is to reprove this conjecture using purely algebraic techniques. In particular, the behavior near infinity was analyzed without relying on calculus, instead utilizing the Newton polygon.

\begin{thm}\label{thm: unicritical_Silverman_Conj}
    Let $f_c(z) = z^d + c$ for an integer $d\geq 2$.
    Then for all integers $n, m$ with $m | n$ and  $m < n$,
    there is a polynomial $\Psi_{n,m}(t) \in \Z[t]$ such that
    \[
        \Delta_{n,m}(c) = \Psi_{n,m}(d^d c^{d-1}).
    \]
    Moreover, the polynomial $(-1)^{e_{n,m}}\Psi_{n,m}(t)$ is monic, where $d_m = \sum_{k|m} d^k \mu\left(\frac{m}{k}\right)$, $\varphi(n)$ is Euler's totient function, and
    \begin{align}
        e_{n,m} \coloneq
        \begin{dcases}
            \varphi(n) \frac{d_m}{m} & \text{ if } n \neq m \\
            \frac{d_n}{n} + \varphi(n) \sum_{m \mid n, m < n} \frac{d_m}{m} & \text{ if } n = m .
        \end{dcases}
    \end{align}
\end{thm}

When $n > 2$, we may ignore the contribution of terms with $\varphi(n)$ since $\varphi(n)$ is even.
\cref{thm: unicritical_Silverman_Conj} is an immediate consequence of the following theorem.
\begin{thm}[{$\risingdotseq$ \cite[Proposition 1.51]{Hug21}} ]\label{thm: unicritical_main}
    Let $f_c(z) = z^d + c$ for an integer $d\geq 2$.
    Then for all positive integers $m$, the multiplier polynomial $\delta_m(x)$ is in the ring $\Z[d^dc^{d-1},x]$
    and $(-1)^{\frac{d_m}{m}}\delta_m(x)$ is monic in $d^d c^{d-1}$.
\end{thm}
\begin{rem}\label{rem: difference_of_proofs}
\cref{thm: unicritical_main} is the same statement as Proposition 1.51 of Huguin's thesis \cite{Hug21} except for the explicit description of the sign.
Our proof of \cref{thm: unicritical_main} can be divided into two parts.
The first is the integrality of $\delta_m$, and the second is the monicness of $\delta_m$ in $d^d c^{d-1}$.
The first one is essentially the same as Huguin's proof \cite[Affirmation 1.52]{Hug21}.
Our proof of the monicness relies on Newton polygons, contrary to Huguin's proof, which uses the analysis around the point at infinity.
We can adapt the proof to the other classes of polynomial maps than unicritical polynomials as in \cref{thm: non_unicritical_main}.
\end{rem}

Fix a parameter $\gamma \in \C$. 
A periodic point $\alpha \in \C$ of $f_\gamma(z) = z^d + \gamma$ of period $m$ is \textit{parabolic} if $\omega_m(\alpha)$ is a root of unity.
A parameter $\gamma$ is said to be a \textit{parabolic parameter} if $f_\gamma$ has a parabolic periodic point.
Note that $\gamma$ is a parabolic parameter if and only if it is a root of $\Delta_{n,m}$ for some $n, m$ with $m|n$.
We investigate the arithmetic properties of parabolic parameters.

For $\gamma \in \overline{\Q}$, denote $H(\gamma)$ its absolute height as in \cref{subsec: number_fields}.
Combining \cref{thm: unicritical_Silverman_Conj} and the bound on the Multibrot set,
we get an upper bound on the height of parabolic parameters $c$ as follows.
\begin{thm} \label{thm: height_bound_unicritical}
    If $\gamma \in \C$ is a parabolic parameter, then $\gamma$ is in $\overline{\Q}$ and
    the inequality $H(\gamma) \le (2 d^d)^{1/(d-1)}$ holds.
\end{thm}

The following conjecture is a central problem in the arithmetic of dynamics.
\begin{conj}[{\cite{Morton-Silverman}}]\label{conj: Morton-Silverman}
    Let $N$ and $D$ be positive integers and $d\geq 2$ be an integer. Let $K$ be a number field.
    Then there is a constant $C(N,D,d,K)$ such that for any morphism $f\colon \bbP^N \longrightarrow \bbP^N$ of degree $d$ defined over $K$,
    the set
    \[
        \left\{ x \in \bbP^N(L)\ \middle|\ [L: K]\leq D, x\text{ has a finite }f \text{-orbit}\right\}
    \]
    has at most $C(N,D,d,K)$ elements.
\end{conj}
Assuming a generalization of the $abc$-conjecture called the $abcd$-conjecture, \cref{conj: Morton-Silverman} is proved in \cite{Looper21a} for unicritical polynomials, and in \cite{Looper21b} for general polynomials.
Even for quadratic polynomial maps, \cref{conj: Morton-Silverman} is still widely open.
The non-existence of rational periodic points of exact period $4$ or $5$ is shown in \cite{Morton98} and \cite{FPS97}, respectively.
Assuming the Birch--Swinerton-Dyer conjecture, the same statement for the exact period $6$ is proved in \cite{Stoll08}.

In this paper, we consider \cref{conj: Morton-Silverman} by focusing on parabolic orbit instead of finite orbit.
%As a consequence of \cref{thm: unicritical_Silverman_Conj}, we get a contribution to \cref{conj: Morton-Silverman}.
By \cref{thm: height_bound_unicritical} and Northcott property of the height function, the set of the parabolic parameter $c$ of degree at most $D$ is finite for all $D\in \Z_{>0}$. 
However, there are two difficulties to determine all parabolic parameters:
\begin{itemize}
    \item The number of candidates of defining polynomials of parabalic parameters with a bounded height increases exponentially fast as $D$ get large.
    \item Even if a candidate fixed, it is difficult to check that it is actually a parabolic parameter when it is close to the Mandelbrot set.
\end{itemize}
Using ideas described in \cref{sec: determination_of_parabolic_parameters}, we determined all quadratic parabolic parameters of $z^2 + c$ as follows.

\begin{thm} \label{thm: parabolic_rational_parameters}
    For $\gamma \in \overline{\Q}$ of degree at most $2$ over $\Q$, assume that $f_{\gamma}(z)= z^2+\gamma$ has a parabolic periodic point in $\C$.
    Then $\gamma$ is either
    \begin{align}
        & -\frac{7}{4}, -\frac{5}{4}, -\frac{3}{4}, \frac{1}{4},
        \frac{1 \pm 2\sqrt{-1}}{4}, \frac{-4 \pm \sqrt{-1}}{4},\\
        & \frac{-1 \pm 3\sqrt{-3}}{8}, \frac{3 \pm \sqrt{-3}}{8}, \frac{-7 \pm \sqrt{-3}}{8}, \text{ or } \frac{-9 \pm \sqrt{-3}}{8}.
    \end{align}
    Moreover, if $\Delta_{n,m}$ is irreducible as a polynomial in $c$ for all $m,n$ with $m \mid n$, the number of parabolic parameters $\gamma \in \C$ of a small degree over $\Q$ is as the following table.
    \begin{table}[h]
    \centering
    \begin{tabular}{|c|c|c|} \hline
        degree &  number of parabolic parameters & total \\ \hline
        $1$ & $4$  & $4$ \\
        $2$ & $12$ & $16$ \\
        $3$ & $6$  & $22$ \\
        $4$ & $32$ & $54$ \\
        $5$ & $0$  & $54$ \\ 
        $6$ & $72$ & $126$ \\
        $7$ & $0$  & $126$ \\ \hline
    \end{tabular}
    \caption{number of parabolic parameters of given degree (under irreducibility conjecture)}
    \label{tab: my_label}
\end{table}
\end{thm}
Note that the irreducibility of $\Delta_{n,m}$ is conjectured by Morton in \cite{Morton-Vivaldi}.
\begin{conj}[The irreducibility conjecture]\label{conj: irreducibility}
    For a family of the polynomial maps $f_c(z) = z^2+c$, the polynomials $\Delta_{n,m}$ in $c$ for positive integers $m,n$ with $m \mid n$ are all irreducible over $\Q$.
\end{conj}
\begin{rem}
    The parabolic parameters of degree at most $2$ are plotted in \cref{fig: Parabolic_parameters_of_degree_at_most_2}.
    \begin{figure}[ht]
	\centering
	\includegraphics[width = 12cm]{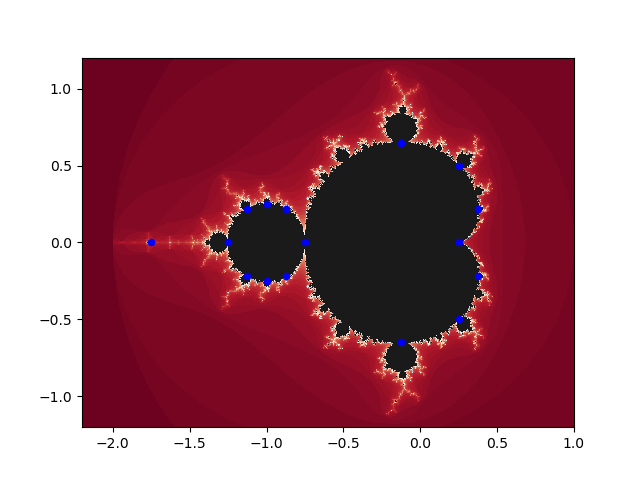}
    \caption{Parabolic parameters of degree at most 2.}
	\label{fig: Parabolic_parameters_of_degree_at_most_2}
\end{figure}
\end{rem}
Since the set of totally real parabolic parameters is completely determined in \cite{BK22}, we essentially determine imaginary quadratic parabolic parameters.
%In \cite{Okuyama22}, readers can find the picture of Mandelbrot set for $z^2+\gamma$ with plots of parameters admitting parabolic periodic points in $\C$.

As mentioned in \cref{rem: difference_of_proofs}, by employing the Newton polygon, we are able to obtain similar integrality results for a broader class of polynomial families.
We prove a similar theorem for $z^{d+1}+ cz$ for $d\geq 1$.
This is the first example of a family of non-unicritical polynomials for which a kind of integrality and monicness of multiplier polynomials hold. 
\begin{thm} \label{thm: non_unicritical_main}
    Let $f_c(z) = z^{d+1} + cz$. For a positive integer $ m $, the polynomial $ d^{\deg_z \Phi_m^*/ (d+1)} \delta_m(x) $ is an element of $ \Z[dc, x] $ and monic in $ dc $ up to multiplication by $ \pm 1 $.
\end{thm}
\begin{thm} \label{thm: height_bound_non_unicritical}
    Let $f_c(z) = z^{d+1} + cz$. For $\gamma \in \C$, assume that $f_{\gamma}(z)$ has a parabolic periodic point in $\C$.
    Then $\gamma$ is in $\overline{\Q}$ and the inequality $H(\gamma) \le  (2d+2)^{1/d} (d+1) $ holds.
\end{thm}
\begin{rem}
    When we focus on the unicritical polynomial maps, the result \cite[Theorem 1]{Ingram19} tells us that for any $d\geq 2$, $n\geq 1$, there exist positive constants $A$ and $B$ such that for any $\lambda \in \overline{\Q}^{\times}$ and a polynomial $f \in \overline{\Q}[z]$ admitting $n$-periodic points of multiplier $\lambda$, the height $H(f)$ is bounded above by $B H(\lambda)^A$.
    While his original theorem holds for general polynomial maps of degree $d$ and any multipliers values, the constants $A$ and $B$ depend on the period $n$.
    From this perspective, although \cref{thm: height_bound_unicritical} and \cref{thm: height_bound_non_unicritical} are theorems for specific polynomial maps and specific multipliers, they have the advantage that the height bound is uniform over the period.
\end{rem}

It is natural to expect that a similar argument works for any $1$-parameter family of polynomials, but this is misleading.
As the following theorem shows, there are cases where the leading coefficient of $\Delta_{n,m}$ in the parameter grows as $n$ grows.
Thus, the uniform integrality for general $1$-parameter families cannot generally hold.

\begin{thm}\label{thm: other_family}
    For $ f_c(z) = z^{d+2} + cz^2 $, let $\Delta_{n,1}$ be the polynomial in $c$ defined above.
    Then the leading coefficient of $\Delta_{n,1}$ in $c$ is $\pm d^{d\varphi(n)} \cyc_n(2)$.
    In particular, for all $n \geq 7$, the leading coefficient has a prime divisor $q$ such that $q \equiv 1 \bmod n$.
\end{thm}

 Morton--Vivaldi shows the following theorem on the relation between $\Phi_n^*(z)$ and $\Delta_{n,m}$.
\begin{thm}[{\cite[Theorem 2.2]{Morton-Vivaldi}}] \label{thm: Morton-Vivaldi_resultant}
    Let $ f(z) $ be a monic polynomial over an integral domain $ \mathcal{O} $.
    Then, we have
    \[
        \Res_z \left( \Phi_n^*(z), \Phi_m^*(z) \right)
        = \pm \Delta_{n, m}^m
    \]
    for all positive integers $n$ and $m$ such that $n>m$ and $m \mid n$.
\end{thm}

\begin{rem}
    In \cite{Okuyama22}, \cref{thm: Morton-Vivaldi_resultant} is generalized to rational functions on the projective line.
\end{rem}

In the proof of \cref{thm: Morton-Vivaldi_resultant}, Morton--Vivaldi considered the expression
\[
    \Res_z \left( \Phi_m^*(z), \Phi_k^*(z) \right)
    = (\text{unit}) \cdot \prod_{\alpha \in Z(\Phi_k^*)} \cyc_{m/k} \left( \left( f^{\circ k} \right)' (\alpha) \right)
\]
in \cite[p. 575, line 14]{Morton-Vivaldi},
where $\cyc_{n}$ is the $n$-th cyclotomic polynomial.
In connection, we give the following equalities on dynatomic modular curves.

\begin{prop} \label{prop: equality_main}
	Let $ f(z) $ be a monic polynomial over an integral domain $ \mathcal{O} $ and $ k $ and $ m $ be positive integers such that $ k \mid m $ and there exists a prime number $p \mid m$ with $p \nmid k$.
	Let $ \alpha \in \overline{\Frac (\mathcal{O})} $ be a periodic point of $ f $ of exact period $ k $ and $ \lambda \coloneqq \omega_k(\alpha) $.
	We denote the $ k $-part and the prime to $ k $-part of $ m $ by
	\[
    	\widetilde{m} \coloneqq \prod_{p \mid k} p^{v_p(m)}, \quad
    	m' \coloneqq \prod_{p \nmid k} p^{v_p(m)} = \frac{m}{\widetilde{m}} > 1
	\]
	respectively, where $ v_p $ is the $ p $-adic valuation with respect to a prime number $ p $.
	Then, we have
	\[
    	\prod_{d \mid \widetilde{m}} \Phi_{f, dm'}^* (\alpha)
    	= \Phi_{f^{\circ \widetilde{m}}, m'}^* (\alpha)
    	= \cyc_{m'} \left( \lambda^{\widetilde{m}/k} \right).
	\]
\end{prop}

%The sign comes from the ring theoretical argument in the \cref{thm: Morton-Vivaldi_resultant}.
%We determined the sign as follows.
%\begin{prop}\label{prop: sign_of_Morton-Vivaldi}
%    \Kaoru{sign of Morton--Vivaldi}
%\end{prop}

\subsection*{Organization of this paper}
\cref{sec: preliminaries} is devoted to preparing notation and some fundamental facts on valuations and complex dynamics.

We prove the theorems on unicritical polynomials $z^d+c$ in \cref{sec: unicritical}.
The fact that $\delta_m(x)$ is in the ring $\Z[d^d c^{d-1}, x]$ is proved in \cref{subsec: unicritical_Integrality}.
In \cref{subsec: unicritical_valuations}, we calculate the $c^{-1}$-adic valuation of roots of $\Phi_{n}^\ast$ and their multipliers through the discussion of Newton polygon.
Using these calculations, we show the monicness of $\delta_m(x)$ in $d^dc^{d-1}$ in \cref{subsec: unicritical_monicness}.
In \cref{subsec: unicritical_integrality_with_fixed_multipliers}, we prove more general results than \cref{thm: unicritical_Silverman_Conj} as a corollary of the results established in the previous sections.
%%%
We prove the theorems on non-unicritical polynomials $z^{d+1} + cz$ in \cref{sec: non_unicritical}.
In \cref{subsec: non_unicritical_replace}, we reduce the assertion to \cref{prop: non_unicritical_resultant} on the map $\tilde{f}(z) = z^{d+1} - cz^d +c$ by dividing by the automorphisms of $z^{d+1}+cz$.
\cref{prop: non_unicritical_resultant} asserts that $d^{m((d+1)^{k-1}-1)/d}R_{k,m}(x)$ is in the ring $\Z[dc,x]$ and is monic in $dc$ up to multiplication by $\pm 1$,
where $R_{k,m}(x)$ is the resultant of some polynomials.
In \cref{subsec: non_unicritical_integrality}, we show that $d^{m((d+1)^{k-1}-1)/d}R_{k,m}(x)$ is in the ring $\Z[dc,x]$.
In \cref{subsec: non_unicritical_valuation}, we calculate the valuations of quantities involving the $n$-th periodic points of $\tilde{f}$.
In \cref{subsec: non_unicritical_monic}, we prove the statement on the monicness of \cref{prop: non_unicritical_resultant} 

%%%
In \cref{subsec: height_bound_unicritical} and \cref{subsec: height_bound_non-unicritical}, we give the upper bound on the heights of the parabolic parameters of the polynomial maps $z^d+c$ and $z^{d+1}+az$, respectively.
We estimate the degree of $\Delta_{n,m}$ in \cref{subsec: deg_estimation}.
Using these height bounds and estimation and case-by-case analysis, the parabolic parameters of degree at most $2$ for $z^2+c$ are determined in \cref{subsec: unconditional_determination_of_parabolic_parameters}.
Under the conjecture that $\Delta_{n,m}$ are irreducible for all $n,m$, we give the number of parabolic parameters of small degrees over $\Q$ in \cref{subsec: Parabolic parameters under the irreducibility conjecture}.

%%%
%The remaining parts of this paper are appendixes.
%We prove \cref{prop: sign_of_Morton-Vivaldi} in \cref{sec: sign_of_Morton-Vivaldi}.
%%%

We prove \cref{prop: equality_main} in \cref{sec: equality}.
%%%
We give lists of computation of $\delta_m$ and $\Delta_{n,m}$ for some polynomials in \cref{sec: list_of_delta_and_Delta}.
%%%
In \cref{sec: Newton_polygon}, we give the complete description of Newton polygon of some polynomials, parts of which are proved in \cref{subsec: non_unicritical_valuation} and used in \cref{subsec: non_unicritical_monic}.
%%%
%We prove \cref{prop: sign_of_Morton-Vivaldi} in \cref{sec: sign_of_Morton-Vivaldi}.

% --------------------------------------------------------------------------

\section*{Acknowledgement}

% --------------------------------------------------------------------------

The authors would like to thank Takashi Taniguchi, Kazunari Sugiyama, and Yasuhiro Ishitsuka, who organized Number Theory Summer School 2023 and gave them a chance to discuss the topic of this paper.
They also would like to thank Rob Benedetto, Valentin Huguin, Ben Hutz, Patrick Ingram, Y\^usuke Okuyama, Joe Silverman, Takayuki Watanabe, and Geng-Rui Zhang for giving some comments on the recent research of parabolic parameters.
They must express our gratitude to the anonymous referee for providing numerous comments and suggestions for improvements on the earlier version of this paper.
The first author would like to thank Masanobu Kaneko for giving many comments.
The second author would like to thank Hiroyuki Inou for giving some knowledge on numerically verified computation.
The third author would like to thank Nobuo Tsuzuki for his helpful comments.
The first author is supported by JSPS KAKENHI Grant Number JP23KJ1675.
The second author was supported by JSPS KAKENHI Grant Number JP20K14300 until March 2024.
The third author is financially supported by JSPS KAKENHI Grant Number JP 22J20227 and the WISE Program for AI Electronics, Tohoku University.

% --------------------------------------------------------------------------

\section{Preliminaries}\label{sec: preliminaries}

% --------------------------------------------------------------------------
Let $K$ be a field and $F, G$ be polynomials over $K$.
Denote $Z(F)$ the multiset of its roots in $\overline{K}$ counted with multiplicity.
Write $F(z) = \sum_{i=0}^n a_i z^i$ and $G(z) = \sum_{j=0}^m b_j z^j$ with $a_n \neq 0$ and $b_m\neq 0$,
Then, the resultant of $F$ and $G$ with respect to $z$ is the quantity
\begin{align}
    \Res_z(F,G) &= a_n^m b_m^n \prod_{\alpha\in Z(F)}\prod_{\beta\in Z(G)}(\alpha - \beta)\\
    &= a_n^m \prod_{\alpha \in Z(F)} G(\alpha).
\end{align}

\begin{lem} \label{lem: Res}
    Let $F$ and $G$ as above.
    \begin{enumerate}
    \item \label{item: lem: Res: det}
    The quantity $\Res_z(F, G)$ is equal to the following determinant of $(m+n)\times(m+n)$ matrix
    \[
    \begin{vmatrix}
        a_n     & a_{n-1}   & \cdots    & \cdots    & a_0       &           &           &\\
                & a_n       & a_{n-1}   & \cdots    & \cdots    & a_0       &           &\\
                &           & \ddots    &           &           &           & \ddots    &\\
                &           &           & a_n       & a_{n-1}   & \cdots    & \cdots    & a_0\\
        b_m     & b_{m-1}   & \cdots    & \cdots    & b_0       &           &           &\\
                & b_m       & b_{m-1}   & \cdots    & \cdots    & b_0       &           &\\
                &           & \ddots    &           &           &           & \ddots    &\\
                &           &           & b_{m}     & b_{m-1}   & \cdots    & \cdots    & b_0
    \end{vmatrix},
    \]
    where the first $m$ row vectors are coefficients of $F$, and the latter $n$ row vectors are coefficients of $G$.
    \item \label{item: lem: Res: base_change}
    Let $\phi \in K[z]$ be a non-constant monic polynomial. Then we have
    \[
    \Res_z(F\circ \phi, G \circ \phi) = \Res(F,G)^{\deg \phi}.
    \]
    \end{enumerate}
\end{lem}
% --------------------------------------------------------------------------

\subsection{Fundamental facts on Number fields}
\label{subsec: number_fields}

% --------------------------------------------------------------------------

For a number field $K$, let $M_K$ be the set of standard absolute values on $K$ whose restriction to $\Q$ is a $p$-adic absolute value or the archimedean absolute value $|\cdot|_{\infty}$.
We write $M_K^{\infty}$ for the set of archimedean absolute values of $K$ lying above $|\cdot|_\infty$ and $M_K^0$ for the set of non-archimedean absolute values of $K$ lying above the $p$-adic absolute values of $\Q$.
For $v\in M_K$, let $K_v$ denote the completion of $K$ at $v$.
Let $ n_v $ denote the local degree at $ v $. Specifically, if $ v $ lies above the $ p $-adic valuation associated with a prime $ p $, then $ n_v = [K_v : \mathbb{Q}_p] $. If $ v $ lies above the archimedean absolute value $ |\cdot|_{\infty} $ of $ \mathbb{Q} $, then $ n_v = [K_v : \mathbb{R}] $.
For $a \in K$, we set $\| a \|_v \coloneq |a|_v^{n_v}$.
The (absolute) {\it height function} $H\colon K\longrightarrow \R$ is defined by
\[
H(a) \coloneq \left(\prod_{v\in M_K} \max\{1, \|a\|_v\}\right)^{\frac{1}{[K : \Q]}}.
\]
for $a\in K$.
The following theorems yield the well-definedness of $H$ as a function on $\bbP^1(\overline{\Q})$.
\begin{thm}[{cf \cite[Section II.1]{Langbook94}}]
    Let $L/K/\Q$ be a tower of number fields and $v\in M_K$ be an absolute value on K.
    Then
    \[
        \frac{1}{[L: K]} \sum_{\substack{w\in M_L\\ w|v}} n_w = n_v.
    \]
\end{thm}
\begin{thm}[{cf \cite[Section V.1]{Langbook94}}] \label{thm: product_formula}
    Let $K$ be a number field.
    Then, for all $ a \in K^{\times}$, we have
    \[
        \prod_{v \in M_K} \|a\|_v = 1.
    \]
\end{thm}

The following theorem is a specific version of a well-known theorem called the Northcott property, which is used to prove finiteness theorems in Diophantine geometry.
\begin{thm}\label{thm: Northcott_property}
    Let $K$ be a number field.
    Then for any integer $D\geq 1$ and any constant $B > 1$, the set
    \[
        \left\{ a \in \overline{K} \ |\ [K(a): K] \leq D \text{ and } H(a) < B \right\}
    \]
    is a finite set.
\end{thm}
% --------------------------------------------------------------------------

\subsection{Fundamental facts on Complete valuation fields}
\label{subsec: complete_valuation_fields}

% --------------------------------------------------------------------------
In the latter sections, we use the facts written in this section for $K_0((c^{-1}))$ and the $c^{-1}$-adic absolute value, where $K_0$ is some field and $K_0((t))$ is the field of Laurent series on $t$.
Note that $K_0((t))$ is the fraction field of the ring of formal power series.

Throughout this section, let $K$ be a complete non-archimedean valuation field for an additive valuation $v$.

\begin{thm}[{cf \cite[Chapter II, (4.8) Theorem]{Neukirchbook99}}] \label{thm: Valuations_unique_ext}
    For any algebraic extension $L/K$, a unique extension $w$ of $v$ to $L$ exists.
    Moreover, when $L/K$ has a finite degree $n$, it is given by the formula
    \[
        w(a) = \frac{1}{n} v \left( N_{L/K}(a)\right)
    \]
    for $a\in L$, where $N_{L/K}\colon L \longrightarrow K$ is the norm map.
\end{thm}

\begin{dfn}
    Let $F(x) = \sum_{i=0}^{n} a_i x^i \in K[x]$ be a polynomial satisfying $a_n \neq 0$.
    To each term $a_i x^i$, we associate a point $(i,v(a_i))\in \R^2$, where we regard $(i,\infty)$ when $a_i = 0$.
    The Newton polygon of $F$ is the lower convex envelope of the set of the points
    \[
        \{ (i,v(a_i))\ |\ 0\leq i \leq n\}.
    \]
    When $a_i = 0$ for all $0\leq i \leq k-1$ and $a_{k}\neq 0$, we regard that the Newton polygon contains the line segment of slope $-\infty$ and of horizontal length $k$.
\end{dfn}

\begin{thm}[{cf.~\cite[Chapter II, (6.3) Proposition]{Neukirchbook99}}]
    Let $f(x) = \sum_{i=0}^n a_ix^i \in K[x]$ be a polynomial with $a_n\neq 0$.
    Let $L$ be the splitting field of $f$ over $K$ and $w$ be the unique extension of $v$.
    If a line segment of slope $-m$ and of horizontal length $r$ occurs in the Newton polygon of $f$, then $f(x)$ has precisely $r$ roots $\alpha_1, \ldots, \alpha_r \in L$ with the valuation
    \[
        w(\alpha_1) = \cdots = w(\alpha_r) = m.
    \]
\end{thm}

For $F = \sum_{i = m}^{\infty} a_i t^i \in K_0((t))$ with $a_{m}\neq 0$, we set $v_t(F) = m$.
Then, $K_0((t))$ is complete with respect to $v$.
For a polynomial $F\in K_0[c] \subset K_0((c^{-1}))$, we have
\[
\deg_c(F) = -v_{c^{-1}}(F).
\]
%We need the following lemma to prove \cref{thm: non_unicritical_main} in \cref{sec: non_unicritical}.
We need the following construction to prove \cref{thm: non_unicritical_main} in \cref{sec: non_unicritical}.
Let $ K_0 $ be a field, and $ v $ be an additive valuation on $ K_0 $. 
Then we can construct an additive valuation $ \tilde{v} $ on $ K_0(t) $ by 
\[ \tilde{v}\left( \sum_{i=0}^{n}a_i t^i \right) = \inf_{0 \le i \le n} v(a_i), \]
where $ a_i \in K_0$.

For example, by applying the above construction for $ \overline{\Q((c^{-1}))} \subset \overline{\Q((c^{-1}))}(y) \subset \overline{\Q((c^{-1}))}(x,y) $, we obtain an additive valuation $ v $ on $ \overline{\Q((c^{-1}))}(x,y) $ which satisfies
\begin{itemize}
    \item $ v = -\deg_c $ on $ \Q[c,x,y] $,
    \item $ v(x + f) = \min\{0, v(f)\} $ for all $f \in \overline{\Q((c^{-1}))}(y)$, and 
    \item $ v(y + \alpha) = \min\{0, v_{c^{-1}}(\alpha) \} $ for all $ \alpha \in \overline{\Q((c^{-1}))} $.
\end{itemize}

\subsection{Fundamental facts on dynamics}
\label{subsec: Dynamics}

% --------------------------------------------------------------------------
Let $\alpha\in \C$ is a $n$-periodic point of $f(z)\in \C[z]$.
If the multiplier $\omega_n(\alpha)$ is a root of unity, we say that $\alpha$ is a {\it parabolic periodic point}.
If $|\omega_n(\alpha)| < 1$, we say that $\alpha$ is an {\it attracting periodic point}.
For such a periodic point, the following theorem is well-known.
\begin{thm}[{cf \cite[Theorem 8.6, 10.15]{Milnorbook06}}] \label{thm: parabolic_attracting_critical}
    Let $f$ be a rational map on $\bbP^1$ of degree $d\geq 2$.
    Let $\alpha \in \bbP^1(\C)$ be a parabolic or attracting $n$-periodic point of $f$.
    Then there is a critical point $\beta$ of $f$ satisfying
    \[
    \lim_{m\to\infty} f^{\circ (nm+i)}(\beta)= \alpha
    \]
    for some $0 \leq i \leq n-1$.
\end{thm}
For a polynomial map $f_c(z) = z^{d} + c$ with $d\geq 2$, we say that $\gamma \in \C$ is a {\it parabolic parameter} if $f_{\gamma}$ has a parabolic periodic point in $\C$.
We set
\[
    \mathbb{M}_d \coloneq \{\gamma \in \bbC \colon (f_\gamma^{\circ n}(0))_{n=0}^{\infty} \text{ is bounded}\}.
\]
Sometimes $ \mathbb{M}_d $ is called \textit{Multibrot set}.
By \cref{thm: parabolic_attracting_critical}, parabolic parameters are in $\mathbb{M}_d$.
Moreover, they are known to lie at the nodes of $\mathbb{M}_d$.
% --------------------------------------------------------------------------

\section{Multiplier polynomials for $z^d+c$} \label{sec: unicritical}

% --------------------------------------------------------------------------

In this section, we fix an integer $d \geq 2$ and consider a polynomial map $f_c(z) \coloneqq z^d + c$ with a parameter $c$.
This section aims to prove \cref{thm: unicritical_main}.
We prove that $\delta_m(x)$ $(m\geq 1)$ is in $\Z[d^dc^{d-1}, x]$ in \cref{prop: unicritical_integrality}.
The monicness of $\delta_m(x)$ in $d^d c^{d-1}$ is proved in \cref{prop: unicritical_monicness}.

% --------------------------------------------------------------------------

\subsection{Integrality}
\label{subsec: unicritical_Integrality}

% --------------------------------------------------------------------------
In this subsection, we show the part of the statement in \cref{thm: unicritical_main} that $\delta_m(x)$ is an element of $\Z[d^dc^{d-1}, x]$.
Although the following lemma appears in Huguin's thesis \cite{Hug21}, we describe a proof for the reader's convenience.

\begin{lem}[{\cite[Affirmation 1.52]{Hug21}}]\label{lem: unicritical_integrality_of_differentials}
	For a positive integer  $k$, let $F_k(z) := f_c^{\circ k}(z) - z$. Then, the following statements hold.
	\begin{enumerate}
		\item \label{item: lem: unicritical_conversion_of_indeterminates}
		The polynomial $d^{d^k/(d-1)}F_k(z/d^{1/(d-1)})$ is in the ring $\Z[d^{d/(d-1)}c, z]$ and is monic in $z$.
		\item \label{item: lem: unicritical_integrality_of_differentials}
		For $\alpha \in Z(F_k)$, the element $d\alpha^{d-1}$ is integral over $\Z[d^{d/(d-1)}c]$.
	\end{enumerate}
\end{lem}

\begin{proof}
	To prove the statements of \cref{item: lem: unicritical_conversion_of_indeterminates}, it is enough to show that $d^{d^k/(d-1)}f_c^{\circ k}(z/{d^{1/(d-1)}})$ is in the ring $\Z[d^{d/(d-1)}c, z]$ and is monic in $z$.
	We prove the statements by induction on $k$.
	For $k = 1$, since we have $d^{d/(d-1)} f_c(z/d^{1/(d-1)}) = z^d + d^{d/(d-1)}c$, the statement holds.
	Assume the statement for $k$.
	Then
	\[
	d^{\frac{d^{k+1}}{d-1}} f_c^{\circ (k+1)}\left( \frac{z}{d^{\frac{1}{d-1}}} \right)
	= \left(d^{\frac{d^k}{d-1}} f_c^{\circ k} \left( \frac{z}{d^{\frac{1}{d-1}}}\right)\right)^d + d^{\frac{d^{k+1}-d}{d-1}}\cdot d^{\frac{d}{d-1}}c
	\]
	is in the ring $\Z[d^{d/(d-1)}c, z]$ and is monic in $z$ by the assumption.
	
	\cref{item: lem: unicritical_integrality_of_differentials}.
	For $\alpha \in Z(F_k)$ let $\beta \coloneq d^{1/(d-1)}\alpha$.
	Since we have $d^{d^k/(d-1)}F_k(\beta/d^{1/(d-1)}) = d^{d^k/(d-1)}F_k(\alpha) = 0$, the element $\beta$ is integral over $\Z[d^{d/(d-1)}c]$ by \cref{item: lem: unicritical_conversion_of_indeterminates}.
	Thus $d\alpha^{d-1} = \beta^{d-1}$ is also integral over $\Z[d^{d/(d-1)}c]$.
\end{proof}

\begin{lem}\label{lem: unicritical_integrality_of_R}
	Let $k\geq 1$ and $m\geq 1$.
	For a polynomial $G\in \Z[x, x_0,x_1,\ldots, x_{m-1}]$, define $R_{k}(G) \coloneq \Res_z(f_c^{\circ k}(z) - z, G(x, dz^{d-1}, df_c(z)^{d-1}, \ldots , df_c^{\circ(m-1)}(z)^{d-1}) \in \Z[c, x]$.
	Then, the following statements hold.
	\begin{enumerate}
		\item \label{item: lem: unicritical_integrality_of_coefficients}
		All coefficients of $ R_{k}(G) \in \Z[c][x]$ in the indeterminate $x$ are integral over $\Z[d^{d/(d-1)}c]$.
		\item \label{item: lem: unicritical_invariance_of_coefficients}
		%Regard $G$ as a polynomial in indeterminates $x_0, x_1, \ldots, x_{m-1}$ with coefficients in $\Z[x]$.
		%If the total degree of each term of $G$ is divisible by $d-1$,
		The polynomial $R_{k}(G)$ is in the ring $\Z[d^d c^{d-1}, x]$.
	\end{enumerate}
\end{lem}

%\Kohei{In the previous version of our paper, the definition of $R_k(G)$ was incorrect. If $R_{k}(G) \coloneq \Res_z(f_c^{\circ k}(z) - z, G(x, z, f_c(z), \ldots , f_c^{\circ(m-1)}(z))$, we can't apply \cref{lem: unicritical_integrality_of_differentials}. As a result, under that (wrong) definition, there exists a counterexample. For example, if $m=k=1$ and $G(x,x_0) = x_0$, we have $R_k(G) = \Res_z(f_c^{\circ k}(z)-z,z) = f_c^{\circ k} (0) \in \Z[c]$. The statement in (ii) is also modified.}

\begin{proof}
	\cref{item: lem: unicritical_integrality_of_coefficients}
	By \cref{lem: unicritical_integrality_of_differentials} \cref{item: lem: unicritical_integrality_of_differentials}, $d \alpha^{d-1}$ is integral over $\Z[d^{d/(d-1)} c]$ for all $\alpha \in Z(F_k)$.
	Since all coefficients of $R_{k}(G)$ in $x$ are given with multiplication and addition of these values, they are integral over $\Z[d^{d/(d-1)}c]$.
	
	\cref{item: lem: unicritical_invariance_of_coefficients}
	Let $\zeta = \zeta_{d-1}$ be a $d-1$-th root of unity and $\nu\colon \bbP^1\longrightarrow \bbP^1$ be the multiplication by $\zeta$.
	Then, since the diagram
	\begin{equation}
		\begin{tikzcd}
			\bbP^1 \arrow[r, "\nu"] \arrow[d, "f_c"] & \bbP^1 \arrow[d, "f_{\zeta c}"] \\
			\bbP^1 \arrow[r, "\nu" '] & \bbP^1
		\end{tikzcd}
		%ラベルの直後に'を入れるとラベルのつく場所がひっくり返る
	\end{equation}
	commutes, $\nu$ induces a bijection between $Z(f_c^{\circ k} - z)$ and $Z(f_{\zeta c}^{\circ{k}} - z)$.
	Thus we have
	\begin{align}
		R_{k}(G)|_{c\mapsto \zeta c} &= \prod_{\alpha' \in Z(f_{\zeta c}^{\circ k} - z)} G(x, d\alpha'^{d-1}, df_{\zeta c}(\alpha')^{d-1}, \ldots, df_{\zeta c}^{\circ (m-1)}(\alpha')^{d-1})\\
		&= \prod_{\alpha \in Z(f_c^{\circ k} - z)} G(x, d(\zeta \alpha)^{d-1}, d (\zeta f_c(\alpha))^{d-1}, \ldots, d (\zeta f_c^{\circ (m-1)}(\alpha))^{d-1})\\
		&= \prod_{\alpha \in Z(f_c^{\circ k} - z)} G(x, d\alpha^{d-1}, d f_c(\alpha)^{d-1}, \ldots, df_c^{\circ (m-1)}(\alpha)^{d-1})\\
		%& \text{by the assumption of }G\\
		&= R_k(G).
	\end{align}
	Hence all the coefficients of $R_k(G) \in \Z[c][x]$ as a polynomial of $x$ are elements of $\Z[c^{d-1}]$ and integral over $\Z[d^{d/(d-1)} c]$ by  \cref{item: lem: unicritical_integrality_of_coefficients}.
	Since $\Z[d^{d/(d-1)}c]$ is normal, the polynomial $R_k(G)$ is in $\Z[c^{d-1}][x]\cap \Z[d^{d/(d-1)}c][x] = \Z[d^d c^{d-1},x]$.
\end{proof}

\begin{prop}\label{prop: unicritical_integrality}
	For any positive integer $m \geq 1$, the multiplier polynomial $\delta_m(x)$ is in $\Z[d^dc^{d-1}, x]$.
\end{prop}

\begin{proof}
	Applying \cref{lem: unicritical_integrality_of_differentials} \cref{item: lem: unicritical_integrality_of_coefficients}
	to $G = x - \prod_{i=0}^{m-1} x_i$,
	we see that $\Res_z(f_c^{\circ k}(z)-z, x-(f_c^{\circ m})'(z))$ is in $\Z[d^dc^{d-1}, x]$.
	It is known that $\delta_m(x)^m = \prod_{k|m}\Res_z(f_c^{\circ k}(z)-z, x-(f_c^{\circ m})'(z))^{\mu(m/k)}$ is in $\Z[c,x]$ by \cite{Morton-Vivaldi}.
	Since each factor $\Res_z(f_c^{\circ k}(z)-z, x-(f_c^{\circ m})'(z))$ is a polynomial in $c^{d-1}$ and $x$, $\delta_m(x)^m$ is as well.
	By Gauss' lemma, $\delta_m(x)^m$ is a polynomial in $d^dc^{d-1}$ and $x$. Here, we note that $\Z[c^{d-1},x]\subset \Q(d^d c^{d-1},x) = \Frac(\Z[d^d c^{d-1},x])$.
	Since the ring $\Z[d^d c^{d-1}, x]$ is normal, $\delta_m(x)$ is in this ring.
\end{proof}

% --------------------------------------------------------------------------

\subsection{Valuations of roots}
\label{subsec: unicritical_valuations}

% --------------------------------------------------------------------------

In this subsection, we study valuations of roots of $f_c^{\circ k}(z)-z$ and $\delta_m(x)$.
Our primary tool is Newton polygons.
The content of this section might be well-known to experts, but we write them for the reader's convenience.

Let $v\colon \Q((c^{-1}))\longrightarrow \Z$ be the $c^{-1}$-adic valuation.
We also denote by $v$ the unique extension of $v$ to the algebraic closure $\overline{\Q((c^{-1}))}$.

\begin{lem}\label{lem: unicritical_valuation_of_roots}
	The following statements hold for all positive integers $k$ and $m$.
	\begin{enumerate}
		\item \label{item: lem: unicritical_NP_of_iterates}
		The Newton polygon of $f_c^{\circ k}(z) \in \Q((c^{-1}))[z]$ consists of one line segment of slope $1/d$ connecting the points $(0,-d^{k-1})$ and $(d^k,0)$ as shown in \cref{fig: Newton_polygon_slope_only_1_over_d}.
		\item \label{item: lem: unicritical_NP_of_F}
		The Newton polygon of $f_c^{\circ k}(z) - z \in \Q((c^{-1}))[z]$ consists of one line segment of slope $1/d$ connecting the points $(0,-d^{k-1})$ and $(d^k,0)$ as shown in \cref{fig: Newton_polygon_slope_only_1_over_d}.
		\item \label{item: lem: unicritical_NP_of_R}
		The Newton polygon of $\Res_z(f_c^{\circ k}(z)-z, x - (f_c^{\circ m})'(z)) \in \Q((c^{-1}))[x]$ consists of one line segment of slope $m(d-1)/d$ connecting the points $(0,-m(d-1)d^{k-1})$ and $(d^k,0)$ as shown in \cref{fig: Newton_polygon_slope_only_m_d-1_over_d}.
		\item \label{item: lem: unicritical_NP_of_delta}
		The Newton polygon of $\delta_m(x) \in \Q((c^{-1}))[x]$ consists of one line segment of slope $m(d-1)/d$ connecting the points $(0,-\sum_{k|m} (d-1)d^{k-1}\mu(m/k))$ and $(\frac{1}{m} \sum_{k|m} d^k \mu(m/k),0)$ as shown in \cref{fig: Newton_polygon_slope_only_m_d-1_over_d}.
	\end{enumerate}
\end{lem}

\begin{figure}[bthp]
	\begin{minipage}[bthp]{0.45\linewidth}
		\centering
		\includegraphics[width = 7.5cm]{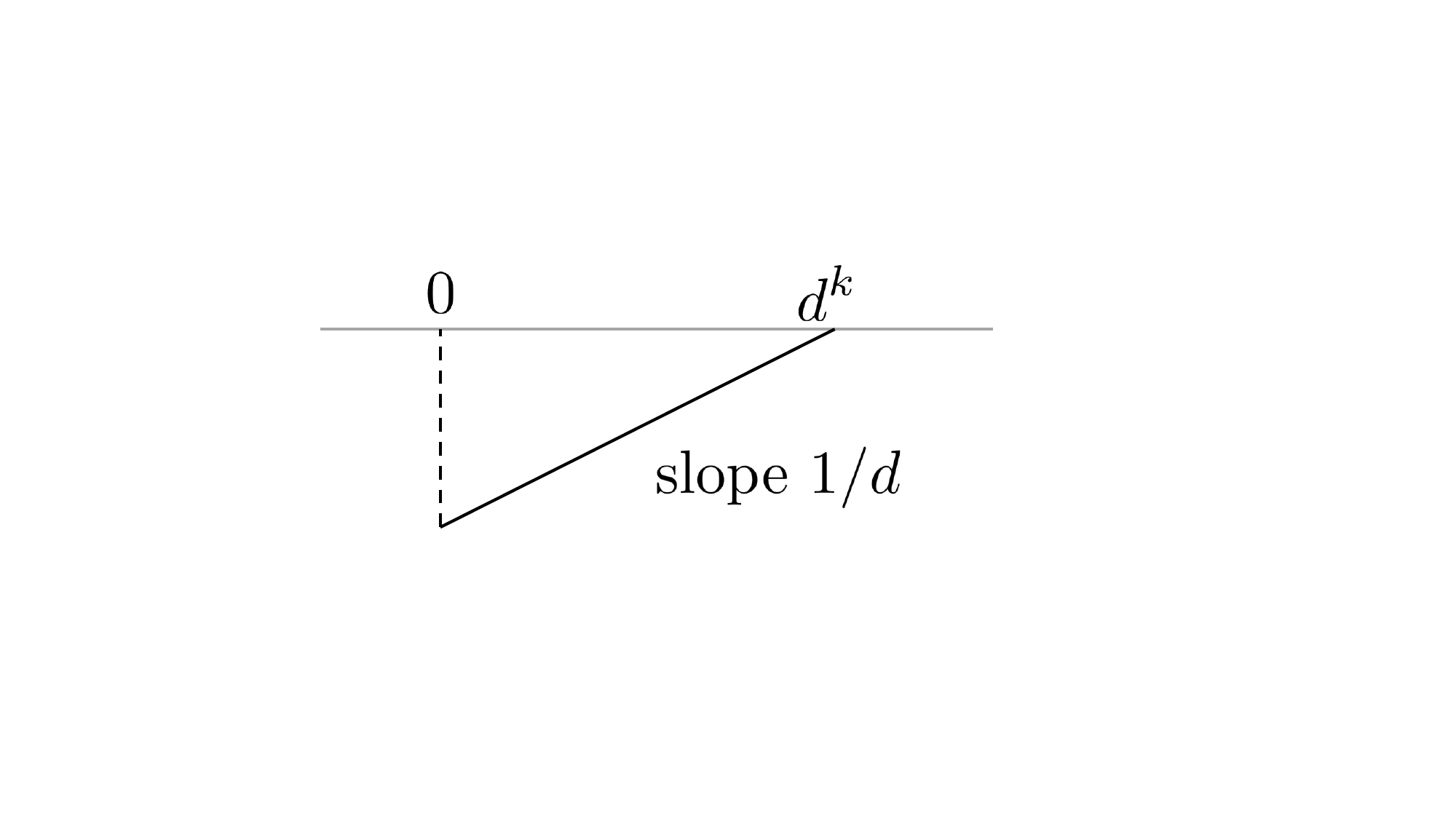}
		\caption{The Newton polygons of $f_c^{\circ k}(z) $ and $f_c^{\circ k}(z) - z$.}
		\label{fig: Newton_polygon_slope_only_1_over_d}
	\end{minipage}		
	\begin{minipage}[bthp]{0.45\linewidth}
		\centering
		\includegraphics[clip,width = 7.5cm]{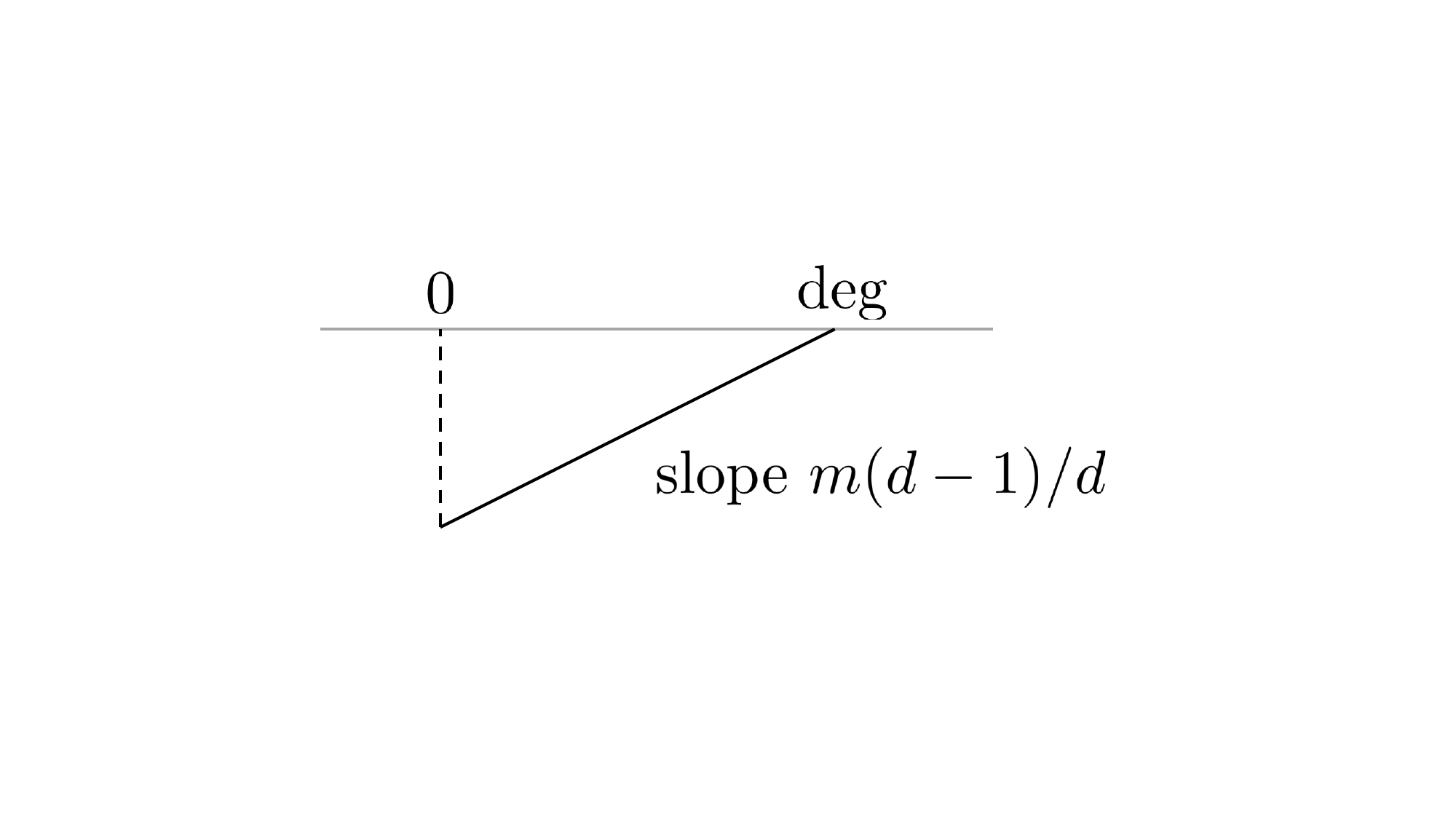}
		\caption{The Newton polygons of $\Res_z(f_c^{\circ k}(z)-z, x - (f_c^{\circ m})'(z)) $ and $\delta_m(x)$.}
		\label{fig: Newton_polygon_slope_only_m_d-1_over_d}
	\end{minipage}
\end{figure}

\begin{proof}
    Note that $f^{\circ k}(z)$ and $f^{\circ k}(z)-z$ are monic in $z$, and $\Res_z(f_c^{\circ k}(z)-z, x - (f_c^{\circ m})'(z))$ and $\delta_m(x)$ are monic in $x$. Thus, the endpoints of Newton polygons of these polynomials are on the axis. 
    
	\cref{item: lem: unicritical_NP_of_iterates} We prove that the Newton polygon of $f_c^{\circ k}$ consists of one line segment of slope $1/d$ by induction on $k$.
	This is trivial for $k=1$.
	Assume the statement for $k$.
	Then since $Z((f_c^{\circ k}(z))^d)$ is a $d$ copy of $Z(f_c^{\circ k})$, the Newton polygon of $(f_c^{\circ k}(z))^{d}$ consists of one line segment of slope $1/d$.
	The point $(0, v(c)) = (0,-1)$ is above the Newton polygon of $(f_c^{\circ k})^d$, and thus the Newton polygon of $f_c^{\circ(k+1)}$ is the same as that of $(f_c^{\circ k})^d$.
	
	\cref{item: lem: unicritical_NP_of_F}
	Since the point $(1, v(1)) = (1,0)$ is above the Newton polygon by \cref{item: lem: unicritical_NP_of_iterates}, the Newton polygon of $F_k = f_c^{\circ k} - z$ is the same as that of $f_c^{\circ{k}}$.
	
	\cref{item: lem: unicritical_NP_of_R}
	For all $\alpha \in Z(f_c^{\circ k}(z) - z)$, we have
	\begin{align}
		&\phantom{{}={}} v\left((f_c^{\circ{m}})'(\alpha)\right)
		= v\left(\prod_{i=0}^{m-1} d \cdot f_c^{\circ i}(\alpha)^{d-1}\right)\\
		&= \sum_{i=0}^{m-1} (d-1) \cdot v(f_c^{\circ i}(\alpha))
		= -m(d-1)/d, \label{eq: unicritical_valuation_of_multiplier}
	\end{align}
        where note that for all $i\geq 0$, the equality $v(f_c^{\circ i}(\alpha)) = -1/d$ holds by \cref{item: lem: unicritical_NP_of_F} since $f_c^{\circ i}(\alpha)$ for $0\leq i \leq m-1$ are elements in $Z(f_c^{\circ k}(z) - z)$ as well.
	The equation \eqref{eq: unicritical_valuation_of_multiplier} implies the statement.
	
	\cref{item: lem: unicritical_NP_of_delta}
	Since we have
	\[
	\delta_m(x)^m = \Res_z\left(\Phi_m^*(z), x - (f_c^{\circ m})'(z)\right) = \prod_{k|m} \Res_z\left(f_c^{\circ k}(z) - z, x-(f_c^{\circ{m}})'(z)\right)^{\mu(m/k)},
	\]
	the valuation of all the roots of $\delta_m(x)$ are equal to $-m(d-1)/d$ by \cref{item: lem: unicritical_NP_of_R}.
	
\end{proof}

% --------------------------------------------------------------------------

\subsection{Monicness}
\label{subsec: unicritical_monicness}

% --------------------------------------------------------------------------
In this subsection, we show the statement in \cref{thm: unicritical_main} that $\delta_m(x)$ is monic in $d^d c^{d-1}$ up to multiplication by $\pm 1$.
\begin{prop}\label{prop: unicritical_monicness}
	For all positive integers $k, m$, the following statements hold.
	\begin{enumerate}
		\item \label{item: prop: unicritical_monicness_Res}
		The leading term of $\Res_z\left(f_c^{\circ k}(z)-z, x-(f_c^{\circ m})'(z)\right)\in \Z[x,d^dc^{d-1}]$ in $c$ is $(-1)^{d^k(m(d-1)+1}(d^dc^{d-1})^{md^{k-1}}$.
		In particular, it is monic in $d^dc^{d-1}$ up to multiplication by $(-1)^{d^k(m(d-1)+1}$.
		\item \label{item: prop: unicritical_monicness_delta}
		The multiplier polynomial $\delta_m(x) \in\Z[x, d^dc^{d-1}]$ is monic in $d^dc^{d-1}$ up to multiplication by $(-1)^{\frac{d_m}{m}}$.
	\end{enumerate}
\end{prop}
%\Kohei{In the previous version of our paper, the sign in (i) was $(-1)^{(m+1)d^k}$ and this is incorrect.}

\begin{proof}
	\cref{item: prop: unicritical_monicness_Res} By \cref{lem: unicritical_valuation_of_roots}\cref{item: lem: unicritical_NP_of_R}, the highest degree of $\Res_z(f_c^{\circ k}(z)-z, x-(f^{\circ m})'(z))$ in $c$ is realized only in the constant term for $x$.
	The constant term for $x$ is calculated as follows: 
	\begin{align}
		\Res_z(f_c^{\circ k}(z)-z, -(f_c^{\circ m})'(z))
		&= \prod_{\alpha\in Z(f_c^{\circ k}(z)-z)} (-1)\prod_{i=0}^{m-1}d (f_c^{\circ i}(\alpha))^{d-1} \\
		&= (-1)^{d^k}\cdot d^{md^k}\left( \prod_{\alpha \in Z(f_c^{\circ k}(z)-z)} \alpha \right)^{m(d-1)}\\
            &= (-1)^{d^k}\cdot d^{md^k}\left( (-1)^{d^k} f_c^{\circ k}(0) \right)^{m(d-1)}\\
		&= (-1)^{d^k(m(d-1)+1)}d^{md^k} f_c^{\circ k}(0)^{(d-1)m}\\
		&= (-1)^{d^k(m(d-1)+1)}(d^d c^{d-1})^{md^{k-1}} + (\text{terms of lower degrees in } c).
	\end{align}
	Thus, the statement holds.
	
	\cref{item: prop: unicritical_monicness_delta}
        Let $\alpha_1,\alpha_2, \ldots, \alpha_r$ represent the $f_c$-orbits of $Z(\Phi_m^\ast)$, where
        $r = d_m/m$ and $d_m = \deg_z\Phi_m^\ast(z)$.
        By \cref{lem: unicritical_valuation_of_roots} \cref{item: lem: unicritical_NP_of_delta}, the highest degree of $\delta_m(x)$ in $c$ is realized only in the constant term for $x$.
        Hence, the following calculation of the constant term in $x$ implies the assertion.
        \begin{align}
            \delta_m(0)
            &= \prod_{i=1}^{r} (-(f_c^{\circ m})'(\alpha_i))
            = (-1)^r \prod_{i=1}^r \prod_{j=0}^{m-1} \left(d (f_c^{\circ j}(\alpha_i))^{d-1}\right)\\
            &= (-1)^r d^{rm} (-1)^{d_m(d-1)} \Phi_m^\ast(0)^{d-1}
            = (-1)^{r+d_m(d-1)} d^{d_m} \Phi_m^\ast(0)^{d-1}\\
            &= (-1)^{r+d_m(d-1)} d^{d_m} \left(\prod_{k|m} \left.\left(f_c^{\circ k}(z) - z\right)^{\mu(m/k)}\middle|_{z=0}\right.\right)^{d-1}\\
            &= (-1)^{r+d_m(d-1)} d^{d_m} \left(\prod_{k|m} \left(f_c^{\circ (k-1)}(c)\right)^{\mu(m/k)}\right)^{d-1}\\
            &= (-1)^{r+d_m(d-1)}d^{d_m} \left(\prod_{k|m} \left(c^{d^{k-1}}+ (\text{terms of lower degrees in }c)\right)^{\mu(m/k)}\right)^{d-1}\\
            &= (-1)^{r+d_m(d-1)}(d^{d} c^{d-1})^{\frac{d_m}{d}} + (\text{terms of lower degrees in }c)\\
            &= (-1)^{\frac{d_m}{m}+d_m(d-1)}(d^{d} c^{d-1})^{\frac{d_m}{d}} + (\text{terms of lower degrees in }c).
        \end{align}
        Since $d \mid d_m$, $d_m (d-1)$ is even in any cases.
\end{proof}
\begin{rem}
    We need to calculate the leading term of $\delta_m$, not of $\delta_m^m$, to determine the leading term without the multiplication by units.
    Thus, \cref{item: prop: unicritical_monicness_Res} of this proposition is not used in the latter discussion but left in to make our argument parallel to the case of $z^{d+1}+cz$.
\end{rem}

% --------------------------------------------------------------------------

\subsection{Integrality of parameters with periodic points of fixed multiplier}
\label{subsec: unicritical_integrality_with_fixed_multipliers}

% --------------------------------------------------------------------------

In this subsection, we prove the following theorem.
\begin{thm}\label{thm: unicritical_integrality_of_parameter}
    Let $d \geq 2$ be an integer and $f_c(z) = z^d + c$.
    Let $\mathcal{O}$ be an integral domain.
    \begin{enumerate}
    \item \label{item: thm: unicritical_integrality_of_resultant}
    For a monic polynomial $\Phi(x)$ over an integral domain $\mathcal{O}$, the quantity $\Res_x(\Phi(x),\delta_{f_c,m}(x))$ is in $\mathcal{O}[d^d c^{d-1}]$, and if $d\neq 0$ in $\mathcal{O}$, it is monic in $d^d c^{d-1}$ up to multiplication by $\pm 1$.
    \item \label{item: thm: unicritical_integrality_of_parameter}
    For a parameter $\gamma \in \mathcal{O}$,
    assume that $f_{\gamma}(z)$ has an $m$-periodic point $\alpha \in \overline{\Frac \mathcal{O}}$ such that its multiplier $(f_{\gamma}^{\circ m})'(\alpha)$ is integral over $\Z$.
    Then $d^d \gamma^{d-1}$ is integral over $\Z$ as well.
    \end{enumerate}
\end{thm}

\begin{proof}
    \cref{item: thm: unicritical_integrality_of_parameter} follows from \cref{item: thm: unicritical_integrality_of_resultant}.
    We prove the assertion of \cref{item: thm: unicritical_integrality_of_resultant}.
    Let $\Phi(x) = \sum_{i=0}^{N} a_i x^i \in \mathcal{O}[x]$ be a monic polynomial.
    Write $\delta_m(x) = \sum_{i=0}^{d_m/m} b_i x^i$, where $d_m = \sum_{k|m} d^k\mu(m/k)$ and $b_i \in \Z[d^d c^{d-1}]$.
    Then we have the equalities
    \begin{align}
        0 &= \Res_x(\Phi(x), \delta_m(x))\\
        &= \begin{vmatrix}
        1       & a_{N-1}   & \cdots    & \cdots    & a_0       &           &           &\\
                & 1         & a_{N-1}   & \cdots    & \cdots    & a_0       &           &\\
                &           & \ddots    &           &           &           & \ddots    &\\
                &           &           & 1         & a_{N-1}   & \cdots    & \cdots    & a_0\\
        b_{d_m/m} & b_{d_m/m-1} & \cdots    & \cdots    & b_0       &           &           &\\
                & b_{d_m/m}   & b_{d_m/m-1} & \cdots    & \cdots    & b_0       &           &\\
                &           & \ddots    &           &           &           & \ddots    &\\
                &           &           & b_{d_m/m}   & b_{d_m/m-1} & \cdots    & \cdots    & b_0
        \end{vmatrix}
    \end{align}
    By \cref{prop: unicritical_monicness} \cref{item: prop: unicritical_monicness_delta}, the minimum valuation of coefficients of $\delta_m(x)$ for $x$ is realized only in the constant term $b_0$, and $\delta_m(x)\in \Z[d^dc^{d-1},x]$ is monic in $d^dc^{d-1}$ up to multiplication by $\pm 1$.
    Hence $\Res_x(\Phi(x), \delta_m(x))$ is also a monic polynomial in $d^d c^{d-1}$ up to multiplication by $\pm 1$.
\end{proof}

\begin{cor}[{\cite[Th\'{e}or\`{e}me]{Bousch14} for quadratic case, \cite[Theorem 1.1]{Milnor14} for general case}]\label{cor: unicritical_integrality_of_parabolic_parameters}
    For an integer $d\geq 2$ and a parameter $\gamma \in \C$, assume that $f_\gamma (z)= z^d + \gamma$ has a parabolic $m$-periodic point $\alpha \in \C$ for some $m\geq 1$.
    Then $d^d \gamma^{d-1}$ is an algebraic integer.
    In particular, when we further assume that $\gamma$ is rational, then $d^d \gamma^{d-1}$ is an integer.
\end{cor}
\begin{proof}
    Since the multiplier of the parabolic periodic point is a root of unity by definition, it is an algebraic integer by applying \cref{thm: unicritical_integrality_of_parameter} in the case of $\mathcal{O}=\C$.
    The last statement follows from the normality of the ring $\Z$.
\end{proof}

\begin{rem}
    \cref{cor: unicritical_integrality_of_parabolic_parameters} is proved by Bousch \cite[Th\'{e}or\`{e}me]{Bousch14} in the quadratic case and Milnor \cite[Theorem 1.1]{Milnor14} in general case by using a property of integral closures.
    We give an explicit polynomial relation which is satisfied by $d^d \gamma^{d-1}$.
\end{rem}

% --------------------------------------------------------------------------

\section{Multiplier polynomials for $ z^{d+1}+ cz $} \label{sec: non_unicritical}

% --------------------------------------------------------------------------

In this section, we fix a positive integer $ d $ and consider a polynomial map $ f(z) \coloneqq z^{d+1}+ cz $ with an indeterminate $ c $.
We aim to prove \cref{thm: non_unicritical_main}.

% --------------------------------------------------------------------------

\subsection{Replacement of polynomials} \label{subsec: non_unicritical_replace}

% --------------------------------------------------------------------------

Let $ \tilde{f}(z) \coloneqq z^{d+1} - cz^d + c $.
For a positive integer $ m $, define 
\[
F_m(z) \coloneqq \prod_{i=0}^{m-1} \tilde{f}^{\circ i}(z) - 1, \quad
\calF_m (y, z) \coloneqq \prod_{i=0}^{m-1} \left( y +  \tilde{f}^{\circ i}(z) \right),
\]
\[
R_{k,m}(x) \coloneq \Res_z \left( F_k(z), x - (d+1)^m \calF_m \left( -\frac{d}{d+1} c, z \right) \right).
\]
Here we recall that we set $ \tilde{f}^{\circ 0} (z) \coloneqq z $ as in \cref{sec: intro}.
As the following lemma shows, these polynomials are helpful to study multiplier polynomials.

\begin{lem} \label{lem: non_unicritical_tilde_f}
    For a positive integer $ m $, the following statements hold.
    \begin{enumerate}
		\item \label{item: lem: non_unicritical_tilde_f: 1}
		Let $ \bbP^1 = \bbP^1_{\Q(c)} $ be the projective line over $ \Q(c) $ and $ \tau \colon \bbP^1 \longrightarrow \bbP^1 $ be the endomorphism defined by $ \tau(z) \coloneqq z^d + c $.
		Then, the diagram
		\begin{equation}
		  \begin{tikzcd}
				\bbP^1 \arrow[r, "\tau"] \arrow[d, "f"] & \bbP^1 \arrow[d, "\tilde{f}"] \\
				\bbP^1 \arrow[r, "\tau" '] & \bbP^1
			\end{tikzcd}
			%ラベルの直後に'を入れるとラベルのつく場所がひっくり返る
		\end{equation}
		commutes.
		
		\item \label{item: lem: non_unicritical_tilde_f: 2}
		We have 
		\begin{align}
			\tilde{f}^{\circ m} (z) &= 
			(z-c) \left( \prod_{i=0}^{m-1} \tilde{f}^{\circ i}(z) \right)^d + c, \\
			f^{\circ m} (z) &= 
			z \prod_{i=0}^{m-1} \tilde{f}^{\circ i} (\tau(z)), \\
			\left( f^{\circ m} \right)' (z) &= 
			(d+1)^m \calF_m \left( -\frac{d}{d+1} c, \tau(z) \right).
		\end{align}
		
		\item \label{item: lem: non_unicritical_tilde_f: 3}
		For a positive integer $ k $, we have
		\[
		\Res_z \left( f^{\circ k}(z) - z, x - \left( f^{\circ m} \right)' (z) \right)
		=
		(x - c^m) R_{k,m}(x)^{d}.
		\]
		In particular, we have 
		\[
		\delta_m(x)^m
		=
		(x - c^{m})^{\varepsilon_m}
		\left(
		\prod_{k \mid m} R_{k,m}(x)^{\mu(m/k)}
		\right)^d,		
		\]
		where
		\[
		\varepsilon_m \coloneqq
		\begin{cases}
			1 & \text{ if } m=1, \\
			0 & \text{ if } m \neq 1.
		\end{cases}
		\]
  
		\item \label{item: lem: non_unicritical_tilde_f: 4}
		The map $ \tilde{f} $ induces a permutation of $ Z(F_m) $, which is the set of roots of $ F_m $.
		
		\item \label{item: lem: non_unicritical_tilde_f: 5}
		The degree of $ F_m(z) $  with respect to $ c $ is $ \frac{(d+1)^{m-1} - 1}{d} $.
		In particular, the polynomial $ d^{\frac{(d+1)^{m-1} - 1}{d}} F_m(z) $ is an element of $ \Z[dc, x] $.
    \end{enumerate}
\end{lem}

\begin{proof}
	\cref{item: lem: non_unicritical_tilde_f: 1} follows from a direct calculation.
	
	\cref{item: lem: non_unicritical_tilde_f: 2}.
	The first equation holds for $ m = 1 $ since $ \tilde{f}(z) = (z-c)z^d + c $.
	If it holds for $ m $, then we have
	\[
	\tilde{f}^{\circ (m+1)} (z) 
	=
	\left( \tilde{f}^{\circ m} (z) - c \right) \tilde{f}^{\circ m} (z)^d + c
	= 
	(z-c) \left( \prod_{i=0}^{m-1} \tilde{f}^{\circ i}(z) \right)^d \tilde{f}^{\circ m} (z)^d + c.
	\]
	Thus, we obtain the first equation.
	
	Next, we prove the second equation.
	By \cref{item: lem: non_unicritical_tilde_f: 1}, we have $ \tilde{f}^{\circ m} \left( \tau(z) \right) = \tau \left( f^{\circ m} (z) \right) = f^{\circ m}(z)^d + c $.
	Thus, by combining the first equation, we obtain
	\begin{align}
		f^{\circ m}(z)^d
		&=
		\tilde{f}^{\circ m} \left( \tau(z) \right) - c
		=
		\left( \tau(z) - c \right) \left( \prod_{i=0}^{m-1} \tilde{f}^{\circ i} \left( \tau(z) \right) \right)^d
		=
		\left( z \prod_{i=0}^{m-1} \tilde{f}^{\circ i} \left( \tau(z) \right) \right)^d.
	\end{align}
	Thus, we obtain the second equation.
	
	Finally, we prove the last equation.
	The chain rule implies $ \left( f^{\circ m} \right)' (z) = \prod_{i=0}^{m-1} f' \left( f^{\circ i}(z) \right) $.
	Since $ f'(z) = (d+1) z^d + c $ and $ f^{\circ m} (z)^d = \tilde{f}^{\circ m} \left( \tau(z) \right) - c $, we have
	\[
	f^{\circ m}(z)' 
	= 
	\prod_{i=0}^{m-1} \left( (d+1) f^{\circ i} (z)^d + c \right)
	= 
	\prod_{i=0}^{m-1} \left( (d+1) \tilde{f}^{\circ i} (\tau(z)) - dc \right)
	= 
	(d+1)^m \calF_m \left( -\frac{d}{d+1}c, \tau(z) \right).
	\]
	
	\cref{item: lem: non_unicritical_tilde_f: 3}.
	By \cref{item: lem: non_unicritical_tilde_f: 2}, we have
	\begin{align}
		&\Res_z \left( f^{\circ k}(z) - z, x - f^{\circ m}(z)' \right)
		\\
		= \,
		&\Res_z \left( z F_k \left( \tau(z) \right), x - (d+1)^m \calF_m \left( -\frac{d}{d+1}c, \tau(z) \right) \right)
		\\
		= \,
		&\left( x - (d+1)^m \calF_m \left( -\frac{d}{d+1} c, \tau(0) \right) \right)
		\Res_z \left( F_k \left( z^d + c \right), x - (d+1)^m \calF_m \left( -\frac{d}{d+1} c, z^d + c \right) \right).
	\end{align}
	Here we have $ (d+1)^m \calF_m \left( -\frac{d}{d+1} c, \tau(0) \right) = c^m $ since 
	$ \tilde{f}^{\circ i} \left( \tau(0) \right) = \tau\left( f^{\circ i}(0) \right) = \tau(0) = c $ for each $ 0 \le i < m $ by \cref{item: lem: non_unicritical_tilde_f: 2}.
	In addition, for any polynomials $ F(z) $ and $ G(z) $, we have 
	$ \Res \left( F\left( z^d + c \right), G\left( z^d + c \right) \right) = \Res(F, G)^d $ by \cref{lem: Res} \cref{item: lem: Res: base_change}.
	Thus, we obtain the first equation.
	The last equation follows from the first equation and the M\"{o}bius inversion formula.
	
	\cref{item: lem: non_unicritical_tilde_f: 4}.	
	By \cref{item: lem: non_unicritical_tilde_f: 1}, we have 
	$ \tilde{f}^{\circ m} (z) - z =	(z-c) \left( \left( \prod_{i=0}^{m-1} \tilde{f}^{\circ i}(z) \right)^d - 1 \right) $.
	Since $ F_m(z) $ divides this, for each $ \alpha \in Z(F_m) $ we have $ \tilde{f}^{\circ m} (\alpha) = \alpha $.
	Thus, we obtain
	\[
	F_m \left( \tilde{f}(\alpha) \right)
	= \prod_{i=1}^{m} \tilde{f}^{\circ i} (\alpha) - 1 
	= \prod_{i=0}^{m-1} \tilde{f}^{\circ i} (\alpha) - 1 
	= F_m (\alpha) 
	= 0.
	\]
	This implies $ \tilde{f}(\alpha) \in Z(F_m) $.
	Since $ \tilde{f}^{\circ m} (\alpha) = \alpha $, $ \tilde{f} $ induces a permutation of $ Z(F_m) $.
	
	\cref{item: lem: non_unicritical_tilde_f: 5} follows from\cref{item: lem: non_unicritical_tilde_f: 2} and induction on $m$.
\end{proof}

By \cref{lem: non_unicritical_tilde_f} \cref{item: lem: non_unicritical_tilde_f: 3}, \cref{thm: non_unicritical_main} follows from the following proposition.

\begin{prop} \label{prop: non_unicritical_resultant}
	For positive integers $ k $ and $ m $, the polynomial $d^{m \frac{(d+1)^{k-1} - 1}{d}} R_{k,m}(x)$ is an element of $ \Z[dc, x] $ and is monic in $ dc $ up to multiplication by $ \pm 1 $.
\end{prop}
In the following subsections, we prove this proposition.

% --------------------------------------------------------------------------

\subsection{Integrality} \label{subsec: non_unicritical_integrality}

% --------------------------------------------------------------------------

The first statement in \cref{prop: non_unicritical_resultant} follows from the following proposition.

\begin{prop} \label{prop: non_unicritical_degree}
	For positive integers $ k $ and $ m $, the degree of $ \Res_z \left( F_k(z), x - \calF_m \left( y, z \right) \right) $ with respect to $ c $ is 
	$ m \frac{(d+1)^{k-1} - 1}{d} $.
	Moreover, we have $d^{m \frac{(d+1)^{k-1} - 1}{d}} R_{k,m}(x) \in \Z[dc, x]$.
\end{prop}

To prove this proposition, we study $ c^{-1} $-adic valuations of roots of $ F_k(z) $.

%We denote by $ v $ the $ c^{-1} $-adic valuation on $ \Q(x, y)(( c^{-1} )) $.
%We also denote by $ v $ the unique extension of it on the algebraic closure $ \overline{\Q(x, y)(( c^{-1} ))} $.

We denote by $ v $ the $ c^{-1} $-adic valuation on $ \overline{\Q(( c^{-1} ))} $. 
Following the method in \cref{subsec: complete_valuation_fields}, we extend $ v $ to $ \overline{\Q(( c^{-1} ))} (x,y) $, and the extension is also denoted by $ v $.
Recall that $ v $ has the following properties.
\begin{enumerate}
    \item $ v = -\deg_c $ on $ \Q[c,x,y] $,
    \item \label{item: valuation_x} $ v(x + f) = \min\{0, v(f)\} $ for all $f \in \overline{\Q((c^{-1}))}(y)$, and 
    \item \label{item: negative_valuation}
    $ v(y + \alpha) = \min\{0, v_{c^{-1}}(\alpha) \} $ for all $ \alpha \in \overline{\Q((c^{-1}))} $.
\end{enumerate}

\begin{lem} \label{lem: non_unicritical_valuation_calF}
	For any positive integers $ k $ and $ m $ and any root $ \alpha \in Z(F_k) $, we have $ v \left( \calF_m (y, \alpha) \right) \le 0 $ and 
	$ v \left( x - \calF_m (y, \alpha) \right) = v \left( \calF_m (y, \alpha) \right) $.
\end{lem}

\begin{proof}
    By the property \cref{item: negative_valuation} above \cref{lem: non_unicritical_valuation_calF}, we have
    \[
    v \left( \calF_m (y, \alpha) \right) = v(y+ \alpha) + v \left( y+ \tilde{f}(\alpha) \right) + \cdots + v \left( y+ \tilde{f}^{\circ (m-1)} (\alpha) \right)\le 0,
    \]
    which is the first statement.
    The last inequality follows from the first inequality and \cref{item: valuation_x}.
\end{proof}

Finally, we prove \cref{prop: non_unicritical_degree}.

\begin{proof}[Proof of $  \cref{prop: non_unicritical_degree} $]
	By \cref{lem: non_unicritical_valuation_calF}, we have
	\begin{align}
		&\phantom{{}={}} -\deg_c \Res_z \left( F_k(z), x - \calF_m \left( y, z \right) \right)\\
		&= v \left( \Res_z \left( F_k(z), x - \calF_m \left( y, z \right) \right) \right) \\
		&= \sum_{\alpha \in Z(F_k)} v \left( x - \calF_m (y, \alpha) \right) \\
		&= \sum_{\alpha \in Z(F_k)} v \left( \calF_m (y, \alpha) \right) \\
		&= v \left( \prod_{\alpha \in Z(F_k)} \prod_{i=0}^{m-1} \left( y+ \tilde{f}^{\circ i} (\alpha) \right) \right).
	\end{align}
	By \cref{lem: non_unicritical_tilde_f} \cref{item: lem: non_unicritical_tilde_f: 4,item: lem: non_unicritical_tilde_f: 5}, this is equals to 
	\[
	v \left( F_k(y)^m \right)
	= -m \deg_c F_k
	= -m \frac{(d+1)^{k-1} - 1}{d}.
	\]
    
    We prove the last statement.
    Since
    \begin{align}
	R_{k,m} (x)
	&= 
	\Res_z \left( F_k(z), (d+1)^m \left( (d+1)^{-m} x - \calF_m \left( -\frac{d}{d+1} c, z \right) \right) \right) \\
	&=
	(d+1)^{m \deg_z F_k} \Res_z \left( F_k(z), (d+1)^{-m} x - \calF_m \left( -\frac{d}{d+1} c, z \right) \right),
    \end{align}
    by the fist statement, we have
    \[
    d^{m \frac{(d+1)^{k-1} - 1}{d}} R_{k,m}(x) 
    \in \Z[(d+1)^{-1}, dc, x] \cap \Z[c, x]
    = \Z[dc, x].
    \]
\end{proof}

% --------------------------------------------------------------------------

\subsection{Valuations of roots} \label{subsec: non_unicritical_valuation}

% --------------------------------------------------------------------------

To prove the last statement in \cref{prop: non_unicritical_resultant}, we study valuations of roots of $ F_k $ in this subsection.
Our primary tool is Newton polygons.

\begin{lem} \label{lem: non_unicritical_valuation_F_k}
	For a positive integer $ k $ and any roots $ \alpha \in Z \left( \tilde{f}^{\circ k} \right) \cup Z \left( F_k \right) $, we have 
	$ v(\alpha) \ge -1 $.
\end{lem}

\begin{proof}
	It suffices to show that Newton polygons of $ \tilde{f}^{\circ k}(z) $ and $ F_k(z) $ have only slope at most $1$.
	
	To begin with, we prove the statement for $ \tilde{f}^{\circ k}(z) $ by induction.
	For $ m=1 $, the Newton polygon of $ \tilde{f}(z) = z^{d+1} - c z^d + c $ is shown in \cref{fig: Newton_polygon_f} and it has only slope at most $1$.
	Suppose the statements for $ \tilde{f}^{\circ i}(z) $ hold for $ 1 \le i \le m $.
	Then, $ (z-c) \left( \prod_{i=0}^{k-1} \tilde{f}^{\circ i} (z) \right)^d $ has only roots with valuations at most $-1$ and thus its Newton polygon takes the form shown in \cref{fig: Newton_polygon_slope1}.
	Thus, Newton polygon of $ \tilde{f}^{\circ k} (z) = (z-c) \left( \prod_{i=0}^{mk-1} \tilde{f}^{\circ i} (z) \right)^d + c $ also takes the form as shown in \cref{fig: Newton_polygon_slope1}.
	
	Finally, we prove the statement for $ F_k(z) $.
	By the statement for $ \tilde{f}^{\circ k}(z) $, the polynomial $ \prod_{i=0}^{k-1} \tilde{f}^{\circ i} (z) $ has only roots with valuations at most $-1$ and thus its Newton polygon takes the form as shown in \cref{fig: Newton_polygon_slope1}.
	Thus, the Newton polygon of $ F_k (z) = \prod_{i=0}^{k-1} \tilde{f}^{\circ i} (z) - 1 $ also takes the form as shown in \cref{fig: Newton_polygon_slope1}.	
\end{proof}

\begin{figure}[bthp]
	\begin{minipage}[bthp]{0.45\linewidth}
		\centering
		\includegraphics[width = 8.0cm]{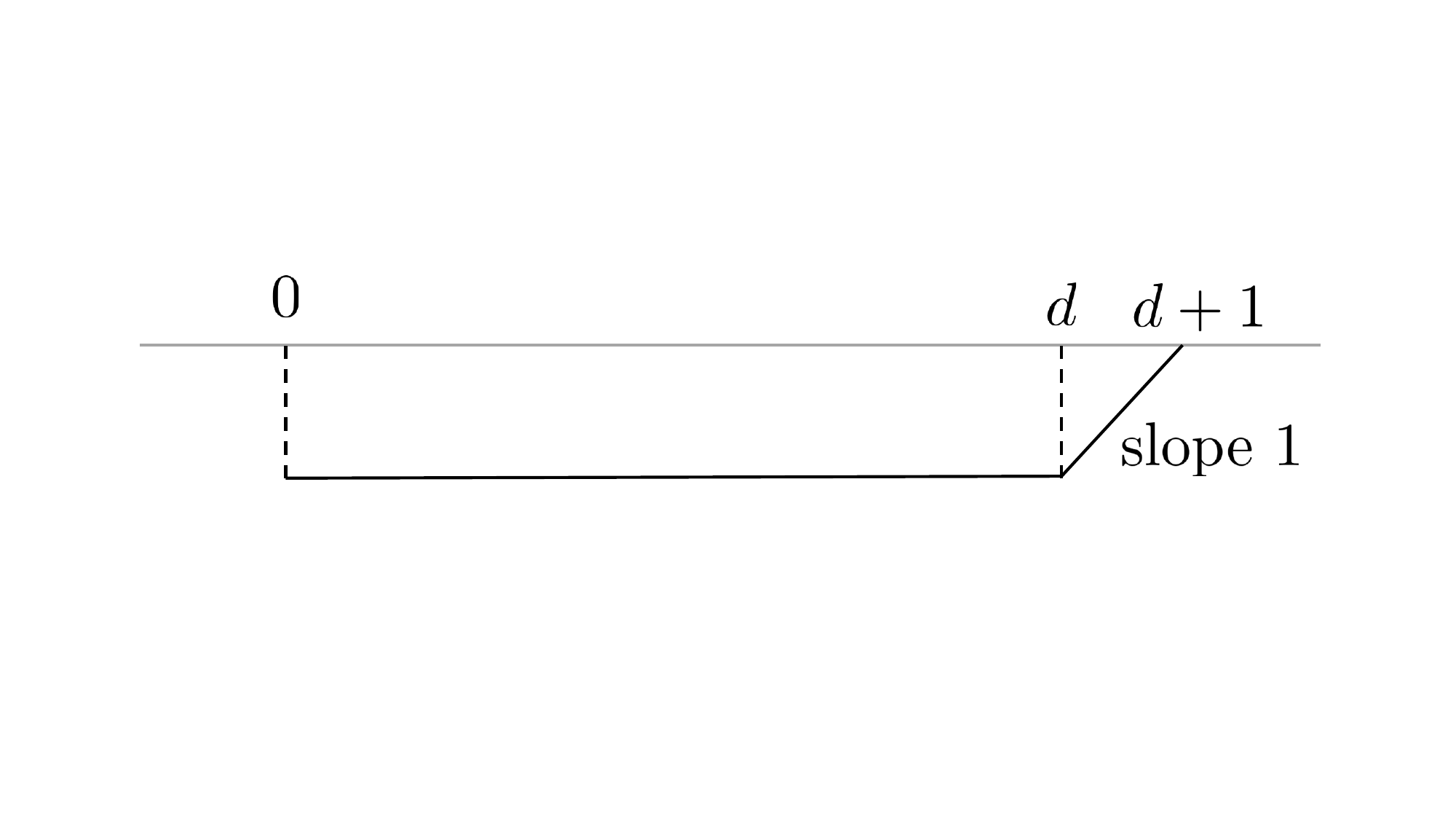}
		\caption{The Newton polygon of $ \tilde{f}(z) $.}
		\label{fig: Newton_polygon_f}
	\end{minipage}		
	\begin{minipage}[bthp]{0.45\linewidth}
		\centering
		\includegraphics[clip,width = 8.0cm]{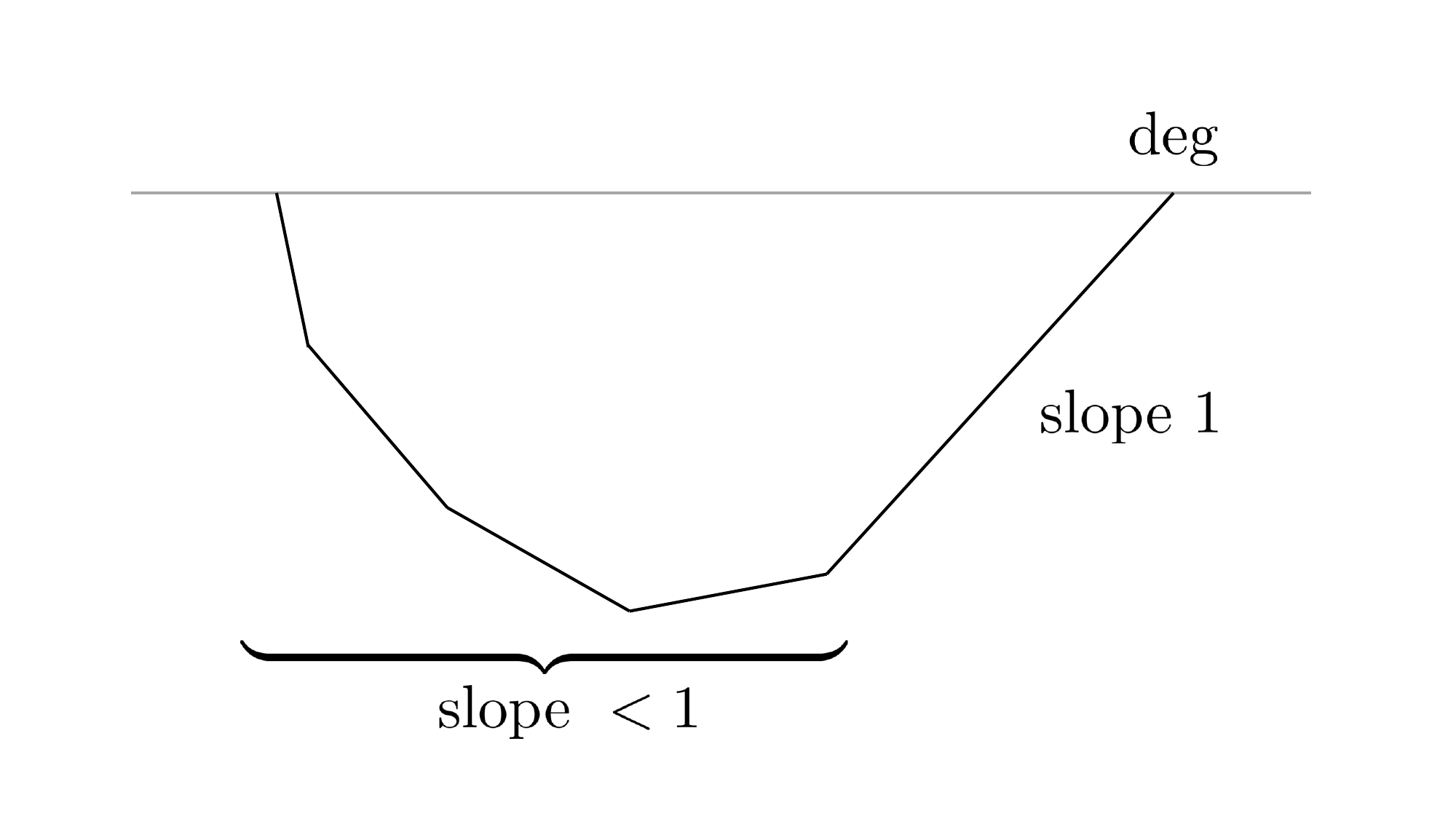}
		\caption{}
		\label{fig: Newton_polygon_slope1}
	\end{minipage}
\end{figure}

\begin{rem} \label{lem: non_unicritical_valuation_f^k}
	We can describe the Newton polygon of $ \tilde{f}^{\circ k} (z) $ explicitly.
	See \cref{fig: Newton_polygon_z_to_d+1_c_z_to_d_+c}.
\end{rem}

\begin{prop} \label{prop: non_unicritical_valuation_F_k}
	For positive integers $ k $ and $ m $ and any roots $ \alpha \in Z \left( F_k \right) $, we have 
	$ v((d+1)\alpha - dc) = -1 $ and $ v \left( (d+1)^m \calF_m \left( -\frac{d}{d+1} c, \alpha \right) \right) = -m $.
\end{prop}

\begin{proof}
    The second equality follows from the first equality because
    \[
        (d+1)^m \calF_m \left( -\frac{d}{d+1} c, \alpha \right)
	=
	\prod_{i=0}^{m-1}  \left( (d+1) \tilde{f}^{\circ i} (\alpha) - dc \right)
    \]
    and $ \tilde{f}(\alpha), \dots,  \tilde{f}^{\circ (m-1)} (\alpha) $ are roots of $ F_k(z) $ by \cref{lem: non_unicritical_tilde_f} \cref{item: lem: non_unicritical_tilde_f: 4}.
	
    We prove the first equality.
    Since $ v(\alpha) \ge -1 $ for any $ \alpha \in Z \left( F_k \right) $ by \cref{lem: non_unicritical_valuation_F_k}, we have 
    \[
	v((d+1)\alpha - dc)
        \ge \min \{ v(\alpha), -1 \}
        = -1.
    \]
    Thus, the Newton polygon of 
    \[
        G_k(x)
        \coloneqq \Res_z \left( F_k(z), x - (d+1)z + dc \right) 
        = \prod_{\alpha \in Z(F_k)} \left( x - (d+1)\alpha + dc \right)
    \]
    is of the form shown in \cref{fig: Newton_polygon_slope1}.
    On the other hand, its constant term equals to
    \[
        G_k(0)
        = \Res_z \left( F_k(z), -(d+1)z + dc \right)
        = (-d-1)^{\deg_z F_k} F_k \left( \frac{d}{d+1} c \right)
        %= (-dc)^{\deg_z F_k} + (\text{ lower terms })
    \]
    and this is a polynomial of $c$ with degree $\deg_z F_k$.
    Thus, its valuation is $ v(G_k(0)) = -\deg_z F_k = -\deg_x G_k(x) $.
    Therefore, the Newton polygon of $ G_k(x) $ takes the form shown in \cref{fig: Newton_polygon_slope_only1}.
    It implies that all roots of $ G_k(x) $ have valuation $ v = -1 $.
\end{proof}

\begin{figure}[bthp]
    \centering
    \includegraphics[clip,width = 8.0cm]{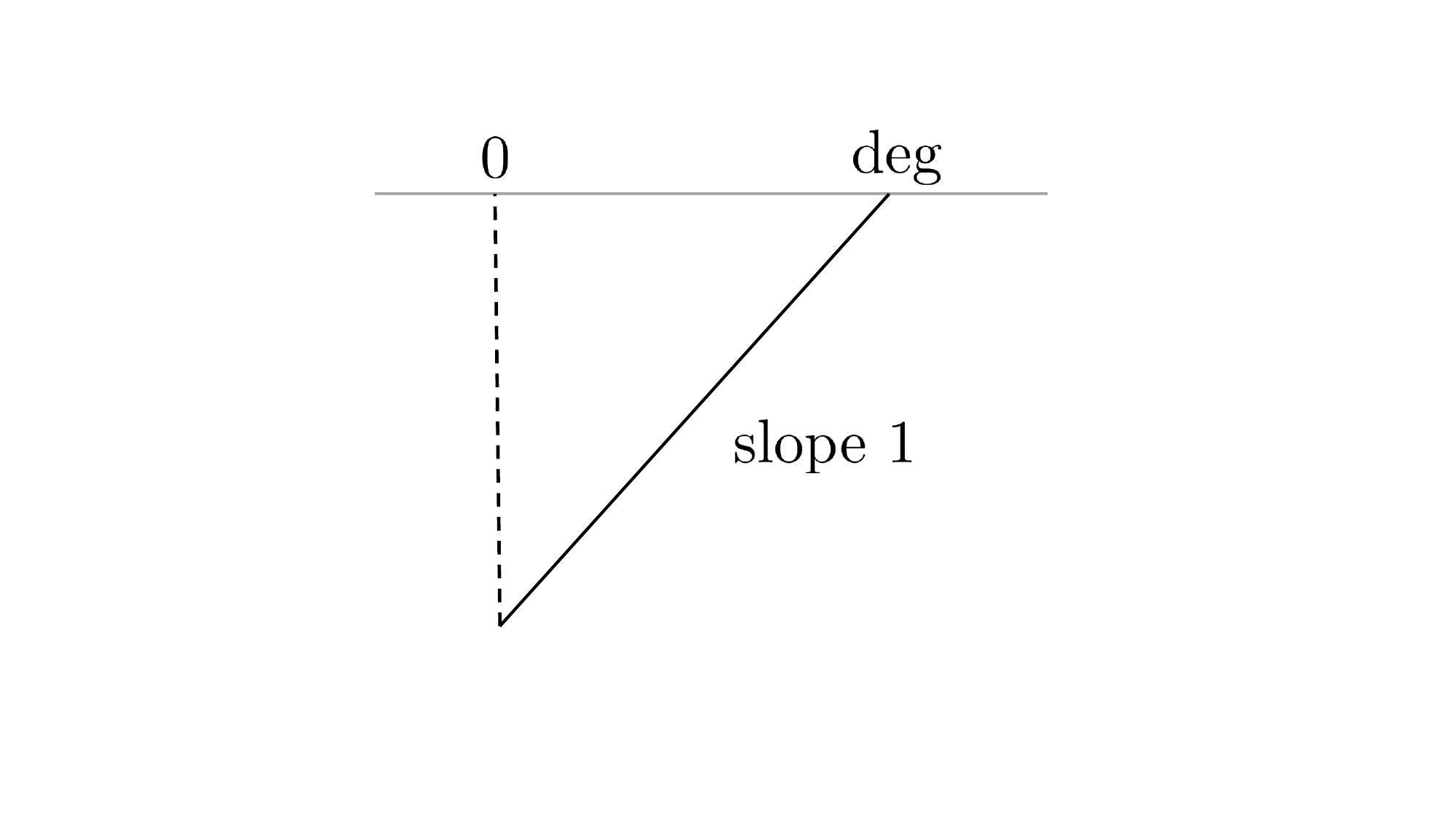}
    \caption{}
    \label{fig: Newton_polygon_slope_only1}	
\end{figure}

% --------------------------------------------------------------------------

\subsection{Monicness} \label{subsec: non_unicritical_monic}

% --------------------------------------------------------------------------

In conclusion of this section, we prove the following proposition. 
\cref{prop: non_unicritical_resultant} follows from \cref{prop: non_unicritical_degree,prop: non_unicritical_leading_term}.

\begin{prop} \label{prop: non_unicritical_leading_term}
    For positive integers $ k $ and $ m $, the leading term of $R_{k,m}(x)$ in $ c $ is
    \[
    (-1)^{\frac{(m+1)((d+1)^{k}-1) + m((d+1)^{k-1} - 1)}{d}} d^{m(d+1)^{k-1}} c^{m \frac{(d+1)^k - 1}{d}}
    \]
    Moreover, $d^{m \frac{(d+1)^{k-1} - 1}{d}} R_{k,m}(x)$ is monic in $ dc $ up to multiplication by $ \pm 1 $.
\end{prop}

Here, we remark that the last statement in this proposition follows from the first statement and \cref{prop: non_unicritical_degree}.
We need the following lemma to prove \cref{prop: non_unicritical_leading_term}.
In the following, we denote the leading term of a polynomial $ P(c) $ by $ \LT_c (P) $.

\begin{lem} \label{lem: non_unicritical_leading_term_f^k}
	For a positive integer $ k $, we have 
	\begin{align}
		\LT_c \left( (d+1)^{\deg_z \tilde{f}^{\circ k}} \tilde{f}^{\circ k} \left( \frac{d}{d+1} c \right) \right)
		&= (-d^d c^{d+1})^{(d+1)^{k-1}}, \\
		\LT_c \left( (d+1)^{\deg_z F_k} F_k \left( \frac{d}{d+1} c \right) \right)
		&= (-1)^{\frac{(d+1)^{k-1} - 1}{d}} d^{(d+1)^{k-1}} c^{\frac{(d+1)^k - 1}{d}}.
	\end{align}
\end{lem}

\begin{proof}
	Let $ \tilde{f}_k \coloneqq (d+1)^{\deg_z \tilde{f}^{\circ k}} \tilde{f}^{\circ k} \left( \frac{d}{d+1} c \right) $ and $ l_k \coloneqq \LT_c \left( \tilde{f}_k \right) $.
	Then, we have
	\[
	l_1
	= \LT_c \left( (d+1)^{d+1} \left( \left( \frac{d}{d+1} c \right)^{d} \left( \frac{d}{d+1} c - c \right) + c \right) \right)
	= -d^d c^{d+1}
	\]
	and
    \begin{align}
        \tilde{f}_{k+1}
	    &= (d+1)^{(d+1) \deg_z \tilde{f}^{\circ k}} \left( \tilde{f}^{\circ k} \left( \frac{d}{d+1} c \right)^{d+1} - c \tilde{f}^{\circ k} \left( \frac{d}{d+1} c \right)^{d} + c \right) \\
	    &= \tilde{f}_k^{d+1} - (d+1)^{\deg_z \tilde{f}^{\circ k}} c \tilde{f}_k^{d} + (d+1)^{(d+1) \deg_z \tilde{f}^{\circ k}} c.
    \end{align}
    These equalities imply $\deg_c \tilde{f}_k \ge 2 $, and thus, 
    \[
        l_{k+1}
        = \LT_c \left( \tilde{f}_{k+1} \right)
        = \LT_c \left( \tilde{f}_{k}^{d+1} \right)
        = l_k^{d+1}.
    \]
	Thus, we obtain $ l_k = (-d^d c^{d+1})^{(d+1)^{k-1}} $.
	Since $ F_k(z) \coloneqq z \tilde{f}(z) \cdots \tilde{f}^{\circ (k-1)} (z) - 1 $ , we have
	\[
	   \LT_c \left( (d+1)^{\deg_z F_k} F_k \left( \frac{d}{d+1} c \right) \right)
	   = \LT_c \left( dc \right) l_1 \cdots l_{k-1}
	   = (-1)^{\frac{(d+1)^{k-1} - 1}{d}} d^{(d+1)^{k-1}} c^{\frac{(d+1)^k - 1}{d}}.
	\]
\end{proof}

\begin{proof}[Proof of $ \cref{prop: non_unicritical_leading_term} $]
	The degree of $R_{k,m}(x)$  with respect to $ c $ is
	\[
	-\sum_{\alpha \in Z(F_k)} v \left( x - (d+1)^m \calF_m \left( -\frac{d}{d+1} c,  \alpha \right) \right).
	\]
	%By \cref{prop: non_unicritical_valuation_F_k,lem: non_unicritical_valuation_trans}, this is equal to
        By \cref{prop: non_unicritical_valuation_F_k}, this is equal to
	\[
	-\sum_{\alpha \in Z(F_k)} (-m)
	=
	m \frac{(d+1)^k - 1}{d}.
	\]
 
	Since $ v(x) = 0 $ and $ v \left( (d+1)^m \calF_m \left( -\frac{d}{d+1} c, \alpha \right) \right) = -m $ for any $ \alpha \in Z(F_k) $ by \cref{prop: non_unicritical_valuation_F_k}, the minimum valuation of the terms of $\displaystyle \prod_{\alpha \in Z(F_k)} \left( x - (d+1)^m \calF_m \left( -\frac{d}{d+1} c, \alpha \right) \right)$ is realized only in the constant term for $x$.
    Hence we have
	\begin{equation}\label{eq: non-unicritical_valuation_with_indeterminates}
            \LT_c \left( \prod_{\alpha \in Z(F_k)} \left( x - (d+1)^m \calF_m \left( -\frac{d}{d+1} c, \alpha \right) \right) \right)
		  = \LT_c \left( \prod_{\alpha \in Z(F_k)} \left( - (d+1)^m \calF_m \left( -\frac{d}{d+1} c, \alpha \right) \right) \right)
	\end{equation}	
	Thus, we can calculate the leading term as
	\begin{align}
    		\LT_c \left( R_{k,m}(x) \right)
    		&=
    		\LT_c \left( \prod_{\alpha \in Z(F_k)} \left( x - (d+1)^m \calF_m \left( -\frac{d}{d+1} c, \alpha \right) \right) \right)\\
    		&= \LT_c \left( \prod_{\alpha \in Z(F_k)} \left( - (d+1)^m \calF_m \left( -\frac{d}{d+1} c, \alpha \right) \right) \right)
    		\text{\qquad \qquad by \eqref{eq: non-unicritical_valuation_with_indeterminates}}\\
    		&=
    		\LT_c \left( \prod_{\alpha \in Z(F_k)} (-1) \prod_{i=0}^{m-1} \left( (d+1) \tilde{f}^{\circ i} (\alpha) - dc \right) \right)\\
    		&= (-1)^{\deg_z F_k} \LT_c \left( \prod_{\alpha \in Z(F_k)} \left( (d+1) \alpha - dc \right) \right)^m \\
    		&=
    		(-1)^{\deg_z F_k} \LT_c \left( (-d-1)^{\deg_z F_k} F_k \left( \frac{d}{d+1} c \right) \right)^m
    		\text{\qquad by \cref{lem: non_unicritical_tilde_f} \cref{item: lem: non_unicritical_tilde_f: 4,item: lem: non_unicritical_tilde_f: 5}}\\
    		&= (-1)^{\deg_z F_k} \left( (-1)^{\deg_z F_k} \cdot (-1)^{\frac{(d+1)^{k-1} - 1}{d}} d^{(d+1)^{k-1}} c^{\frac{(d+1)^k - 1}{d}} \right)^m 
    		\text{\qquad by \cref{lem: non_unicritical_leading_term_f^k}}\\
    		&= (-1)^{\frac{(m+1)((d+1)^{k}-1) + m((d+1)^{k-1} - 1)}{d}} d^{m(d+1)^{k-1}} c^{m \frac{(d+1)^k - 1}{d}},
	\end{align}
        which is what we wanted.
\end{proof}

% --------------------------------------------------------------------------

\section{Height bound on the parabolic parameters} \label{sec: height_bound}

% --------------------------------------------------------------------------

This section establishes a uniform upper bound for the naive height of parabolic parameters in polynomial families $z^d + c$ and $z^{d+1} + cz$.
In the evaluation at finite places, we use the results on integrality discussed earlier, while in the evaluation at infinite places, we rely on the boundedness of Mandelbrot-type sets.

Throughout this section, the set of critical points of a polynomial $ f \in \C[z] $ is denoted by $ \mathrm{Crit}(f) $: 
\[
    \mathrm{Crit}(f) = \{ \alpha \in \C \colon f'(\alpha) = 0 \}.
\]

% --------------------------------------------------------------------------

\subsection{Height bound for $z^d + c$}
\label{subsec: height_bound_unicritical}

% --------------------------------------------------------------------------

In this subsection, we fix $f_c(z)= z^d +c$ with $d\geq 2$.
Recall that the Multibrot set is defined by $ \mathbb{M}_d \coloneq \{\gamma \in \bbC \colon (f_\gamma^{\circ n}(0))_{n=0}^{\infty} \text{ is bounded} \}$.
The following fact might be well-known to experts, but we describe it for the reader's convenience.

\begin{lem} \label{lem: mandelbrot_bound_unicritical}
    Let $\mathbb{M}_d$ be the Multibrot set.
    The set $\mathbb{M}_d$ is contained in the disk $\{c\in \bbC \colon |c| \le 2^{1/(d-1)} \}$.
    In particular, if $f_{c}= z^d + c$ has a parabolic periodic point, then the inequality $|c| \le 2^{1/(d-1)}$ holds.
\end{lem}

\begin{proof}
    The first statement is equivalent to that if $ |c| > 2^{1/(d-1)} $, $\lim_{n \to \infty} |f_c^{\circ n}(0)| = \infty $ holds.
  If $ |c| > 2^{1/(d-1)} $, then there exists a positive  real number $ \varepsilon > 0 $ which satisfies $ |c| \ge (2(1 + \varepsilon))^{1/(d-1)}$. 
  For a complex number $ z \in \bbC $ which satisfies $ |z| \ge |c| $, we have 
    \[ |f_c(z)| \ge |z| \left( |z|^{d-1} - 1 \right) \ge |z| \left( |c|^{d-1} - 1 \right) \ge (1 + \varepsilon)|z|. \]
  By induction, $|f_c^{\circ (n+1)}(0)| = |f_c^{\circ n}(c)| \ge (1 + \varepsilon)^n |c| \to \infty $.
  The second statement is a consequence of $\mathrm{Crit}(f_c) = \{ 0 \} $, the first statement, and \cref{thm: parabolic_attracting_critical}.
\end{proof}

\begin{proof}[Proof of \cref{thm: height_bound_unicritical}]
    Let $ \gamma \in \bbC $ as in the assertion.
  By \cref{thm: unicritical_main}, $ d^d \gamma^{d-1} $ is an algebraic integer, in particular, the parameter $\gamma$ is in $\overline{\Q}$.
  Let $ K = \Q(\gamma) $ be the field generated by $ \gamma $ and $ r = [ K: \Q ] $ be the extension degree.
  Since $ \|d^d \gamma^{d-1}\|_v \le 1 $ for all finite places $ v \in M_K^{0} $, we have
    \[ \prod_{v \in M_K^{0}} \max \{1, \|\gamma\|_v \} \le \prod_{v \in M_K^{0}} \|d^d\|_v ^{-1/(d-1)} = \prod_{v \in M_K^{\infty}} \|d^d\|_v ^{1/(d-1)} = (d^{d/(d-1)})^r .\]
  The second equation follows from the product formula (\cref{thm: product_formula}).
  On the other hand, since the inequality $ |\gamma| \le 2^{1/(d-1)} $ holds by \cref{lem: mandelbrot_bound_unicritical}, we obtain
    \[ \prod_{v \in M_K^{\infty}} \max \{1, \|\gamma\|_v \} \le \left( 2^{1/(d-1)} \right)^r.\]
  By combining these two estimations, 
    \[ H(\gamma)^r  \le (2 d^d)^{r/(d-1)}. \]
\end{proof}

% --------------------------------------------------------------------------

\subsection{Height bound for $z^{d+1} + cz$}
\label{subsec: height_bound_non-unicritical}

% --------------------------------------------------------------------------

In this subsection, we fix the following notation.
\begin{itemize}
  \item $f_c(z) = z^{d+1} + cz$,
  \item $r(d) = (2d+2)^{1/d} \left( 1 + \frac{1}{d} \right) (> 1)$.
\end{itemize}

\begin{lem} \label{lem: mandelbrot_bound_non_unicritical}
	For a positive integer $d$, the following statements hold.
	\begin{enumerate}
		\item \label{item: lem: mandelbrot_bound_non_unicritical_trivial}
        If two complex numbers $c, z \in \bbC$ satisfy $|z| > (2|c|)^{1/d}$ and $ |c| > 1 $, then $\displaystyle \lim_{n \to \infty} |f_c^{\circ n} (z)| = \infty$.
		\item \label{item: lem: mandelbrot_bound_non_unicritical_main}
        For $c \in \C$, assume that $f_c(z) = z^{d+1} + cz$ has a parabolic periodic point, then $|c| \le r(d)$.
	\end{enumerate}
\end{lem}

\begin{proof}
    \cref{item: lem: mandelbrot_bound_non_unicritical_trivial} If two complex numbers $c, z \in \bbC$ satisfy the relation $|z| > (2|c|)^{1/d}$ and $|c| > 1$, then there exists a positive real number $ \varepsilon > 0$ which satisfies $ |z| \ge (2 (1 + \varepsilon) |c|)^{1/d} $. Therefore, 
    \[ |f_c(z)| \ge |z| \left( |z|^d - |c| \right) \ge |z| \left( 2 (1 + \varepsilon) |c| - |c| \right) \ge (1 + \varepsilon) |z|. \]
    By induction, $ |f_c^{\circ n}(z)| \ge (1 + \varepsilon)^n |z| \rightarrow \infty$.
    
    \cref{item: lem: mandelbrot_bound_non_unicritical_main} If $ |c| > r(d) $, then $ d \left| \frac{c}{d+1} \right|^{1+\frac{1}{d}} > (2|c|)^{1/d} $ and $ |c| > 1$.
    On the other hand, the set of critical points of $ f_c $ is described as $\mathrm{Crit}(f_c) = \left\{ \zeta_d^i \left( \frac{-c}{d+1} \right)^{1/d} \colon 0 \le i \le d-1 \right\}$ and $\left|f_c\left( \zeta_d^i \left( \frac{-c}{d+1} \right)^{1/d}\right)\right| = d \left| \frac{c}{d+1} \right|^{1+\frac{1}{d}} $.
    Therefore, every critical orbit of $f_c$ is attracted into $\infty$ by (i). By \cref{thm: parabolic_attracting_critical}, $ f_c $ has no parabolic periodic point.
\end{proof}

\begin{proof}[Proof of \cref{thm: height_bound_non_unicritical}]
    The same argument as in the proof of \cref{thm: height_bound_unicritical} works in this situation.
    Indeed, we can prove 
    \[
        \prod_{v \in M_K^{0}} \max \{1, \|\gamma\|_v \} \le d^r
    \]
    by \cref{thm: non_unicritical_main} and 
    \[
        \prod_{v \in M_K^{\infty}} \max \{1, \|\gamma\|_v \} \le r(d)^r
    \]
    by \cref{lem: mandelbrot_bound_non_unicritical} for each $ \gamma $ such that $f_{\gamma}$ has a parabolic periodic point.
\end{proof}

% --------------------------------------------------------------------------

\section{Determination of parabolic parameters of fixed degrees}
\label{sec: determination_of_parabolic_parameters}

% --------------------------------------------------------------------------

Throughout this section, let $f_c(z) = z^2 + c$.
We estimate the degree of $\Delta_{n,m}$ in \cref{subsec: deg_estimation}.
In \cref{subsec: unconditional_determination_of_parabolic_parameters}, we determine the rational or imaginary quadratic parabolic parameters of $f_c(z)$.
We determine the parabolic parameters of $f_c(z)$ of the fixed degree over $\Q$ in \cref{subsec: Parabolic parameters under the irreducibility conjecture} assuming the irreducibility of $\Delta_{n,m}$ for all $n,m \in \Z_{\geq 1}$ with $m \mid n$.

% --------------------------------------------------------------------------

\subsection{Degree estimate of $\Delta_{n,m}$}\label{subsec: deg_estimation}

% --------------------------------------------------------------------------

Put $ \nu(n) := \sum_{k | n} 2^{k-1}\mu(n/k) $. The first few terms are listed.
\begin{table}[h]
    \begin{tabular}{c|cccccccc} \hline
        $n$ & 1& 2& 3& 4& 5& 6& 7& 8 \\ \hline
        $2^{n-2}$& 1/4& 1/2& 2& 4& 6& 8& 16& 32 \\
        $\nu(n)$& 1& 1& 3& 6& 15& 27& 63& 120 \\ 
        $7 \cdot 2^{n-2}$& 7/4& 7/2& 7& 14& 28& 56& 112& 224 \\ \hline
    \end{tabular}
\end{table}

Then \cite[Corollary 3.3]{Morton-Vivaldi} shows that
\begin{align}
    \deg_c \Delta_{n,m} = 
    \begin{cases}
        \nu(m) \varphi(n/m) & \text{ if } m \mid n, n \neq m, \\
        \nu(n) - \sum_{k|n, k\neq n} \nu(k)\varphi(n/k) & \text{ if } n = m,
    \end{cases} 
\end{align}
where $ \varphi(n) \coloneqq \abs{(\Z/n\Z)^\times} $ is the Euler totient function.

\begin{prop}\label{prop: nu}
    For any positive integer $ n $, the following estimate holds: 
    \[ 2^{n-3} \le \nu(n) \le 7\cdot 2^{n-3}. \]
\end{prop}

\begin{proof}
    The statement is directly verified if $ 1 \le n \le 5 $. 
    Henceforth, we assume $ n \ge 6 $. Since $ k \le n/2 $ if  $k|n, k\neq n$,
    \begin{align}
        \left|\frac{\nu(n)}{2^{n-1}} - 1 \right| = \frac{1}{2^{n-1}} \left| \sum_{k|n, k\neq n} 2^{k-1} \mu\left(\frac{n}{k}\right)\right| \le \frac{n 2^{n/2 - 1}}{2^{n-1}} = \frac{n}{2^{n/2}}.
    \end{align}
    Since $6/2^{6/2} = 3/4$ and the sequence $n/2^{n/2}$ is decreasing if $n \ge 3$, we have 
    \begin{align}
        \left|\frac{\nu(n)}{2^{n-1}} - 1 \right| \le \frac{3}{4}.
    \end{align}
    This inequality is equivalent to
    \[ 2^{n-3} \le \nu(n) \le 7\cdot 2^{n-3}.\]
\end{proof}

The following assertion is well-known; however, it is included here for the reader's convenience.

\begin{prop}\label{prop: euler_phi}
    For any positive integer $ n $, the following estimate holds:
    \[ \sqrt{\frac{n}{2}} \le \varphi(n) \le n. \]
\end{prop}

\begin{proof}
    The upper bound is trivial.
    Let's show the lower estimate. 
    First, we consider the case when $ n $ is odd and let $ n = p_1^{e_1} \cdots p_r^{e_r} $ be the prime factorization.
    Then, since $ x - 1 \ge \sqrt{x} $ for $ x \ge (3+\sqrt{5})/2 = 2.61803398\dots$ and $ e - 1/2 \ge e/2 $ for $e \ge 1$,
    \begin{align}
        \varphi(n) = \prod_{1\le i \le r} p_i^{e_i-1}(p_i - 1) \ge \prod_{1\le i \le r} p_i^{e_i-1/2} \ge \prod_{1\le i \le r} p_i^{e_i/2} = \sqrt{n} \ge \sqrt{\frac{n}{2}}.
    \end{align}
    If $n$ is even, we can write $n=2^e m$ using an integer $e$ greater than or equal to $1$ and an odd number $m$.
    Therefore, using the odd case, we have
    \begin{align}
        \varphi(n) = \varphi(2^e)\varphi(m) = 2^{e-1}\varphi(m) \ge 2^{(e-1)/2}\sqrt{m} = \sqrt{\frac{2^e m}{2}} = \sqrt{\frac{n}{2}}.
    \end{align}        
\end{proof}

\begin{prop}\label{prop: equal_case}
    For a positive integer $n \ge 5$, the following estimate holds:
    \[\nu(n) - \sum_{k|n, k\neq n} \nu(k)\varphi(n/k) \ge 2^{n/2}\]
\end{prop}

\begin{proof}
    The statement is directly verified if $ 5 \le n \le 23 $. 
    Henceforth, we assume $ n \ge 24 $.
    First, note that the function $ F(x) = 2^{x/2} - 7x^2 $ satisfies $ F(x) \ge 8 $ when $x \ge 24$. From this fact, \cref{prop: nu,prop: euler_phi}, we obtain the following estimate.
    \begin{align}
        &\nu(n) - \sum_{k \mid n, k\neq n} \nu(k)\varphi(n/k) \ge 2^{n/2}\\
        \ge & 2^{n-3} - \sum_{k \mid n, k\neq n} 7 \cdot 2^{k-3} \frac{n}{k} \\
        \ge & 2^{n/2 - 3} \left(  2^{n/2} - 7 n^2 \right) \\
        \ge & 2^{n/2}.
    \end{align}
\end{proof}

Next, we determine all pairs $ (n,m) $ such that $ \Delta_{n,m} $ has a given degree.
Put 
\[ d(n) := \deg_c(\Delta_{n,n}(c)) = \nu(n) - \sum_{k|n, k\neq n} \nu(k)\varphi(n/k).\]

\begin{cor}\label{cor: search_range}
    For $a,b \in \mathbb{Z}_{>0}$, the following holds.
    \begin{enumerate}\renewcommand{\labelenumi}{(\arabic{enumi})}
        \item If $\varphi(a) = b$, then $a \le 2b^2$.
        \item If $\nu(a) = b$, then $a \le 3 + \log(b)/\log(2)$.
        \item If $d(a) = b$, then $a \le 5$ or $ a \le 2\log(b)/\log(2) $.
    \end{enumerate}
\end{cor}

\begin{proof}
    Apply \cref{prop: euler_phi}, \cref{prop: nu,prop: equal_case}, respectively.
\end{proof}

\begin{prop}\label{prop: mn_pair}
    For $m,n \in \mathbb{Z}_{>0}$ with $m \mid n$, the following holds.
    \begin{enumerate}\renewcommand{\labelenumi}{(\arabic{enumi})}
        \item If $\deg_c(\Delta_{n,m}(c)) = 1$, then $(m,n) = (1,1),(1,2),(2,4),(3,3)$. 
        \item If $\deg_c(\Delta_{n,m}(c)) = 2$, then $(m,n) = (1,3),(1,4),(1,6),(2,6),(2,8),(2,12)$.
        \item If $\deg_c(\Delta_{n,m}(c)) = 3$, then $(m,n) = (3,6),(4,4)$.
        \item If $\deg_c(\Delta_{n,m}(c)) = 4$, then $(m,n) = (1,5),(1,8),(1,10),(1,12),(2,10),(2,16),(2,20),(2,24)$.
        \item There is no pair $(m,n)$ such that $\deg_c(\Delta_{n,m}(c)) = 5$.
    \end{enumerate}
\end{prop}

\begin{proof}
    Since
    \begin{align}
    \deg_c \Delta_{n,m} = 
    \begin{cases}
        \nu(m) \varphi(n/m) & \text{ if } m \mid n, n \neq m, \\
        d(n) & \text{ if } n = m,
    \end{cases} 
\end{align}
    if $\deg \Delta_{n,m} \le 5$, then $n,m,n/m \le 50$.
    Therefore, it is enough to check a finite number of candidates in \cref{tab: terms}.
\end{proof}

% --------------------------------------------------------------------------

\subsection{Unconditional determination of parabolic parameters}
\label{subsec: unconditional_determination_of_parabolic_parameters}

% --------------------------------------------------------------------------

In this subsection, we unconditionally give the set of rational or quadratic parabolic parameters for $f_c(z)$.
We heavily use computations involving Julia and SageMath.
In particular, we use the package IntervalArithmetic.jl \cite{IntervalArithmetic.jl} of Julia for the numerically verified computation.

We begin with the following lemma, which is a generalization of Huguin's argument written in \cite{BK22} to any number field.
\begin{lem}\label{lem: Huguin's argument}
    Let $K$ be a number field and $\gamma \in K$ be a parabolic parameter for $f_c(z)$.
    Fix a prime ideal $\pe$ of $\mathcal{O}_K$ dividing $2$.
    Then we have $4\gamma \not \equiv 0 \mod \pe$, where note that $4\gamma$ is an algebraic integer by \cref{thm: unicritical_Silverman_Conj}.
\end{lem}
\begin{proof}[Proof of \cref{lem: Huguin's argument}]
    Fix $\gamma \in K$ such that $4\gamma$ is an algebraic integer.
    It is enough to show that if $4\gamma \equiv 0 \mod \pe$, then $\gamma$ is not a parabolic parameter.
    Since we have
    \begin{align}
        \Disc_z(f_\gamma^{\circ n}(z) - z)
        &= \pm\Res_z(f_\gamma^{\circ n}(z) - z, -1 + \omega_n(z))\\
        &= \pm\delta_n(1),
    \end{align}
    there is a monic polynomial $P_n(t) \in \Z[t]$ such that
    \[
    \Disc_z(f_\gamma^{\circ n}(z) - z) = \pm P_n(4\gamma)
    \]
    by \cref{thm: unicritical_main}.
    If $4\gamma \equiv 0 \mod \pe$, since we have
    \begin{align}
        P_n(4\gamma) \equiv P_n(0) = \Disc_z(z^{2^n} - z) \equiv 1 \mod \pe,
    \end{align}
    the parameter $\gamma$ cannot be parabolic.
\end{proof}
    Before the proof of \cref{thm: parabolic_rational_parameters}, we explain how the parabolic parameters are determined in general.
    Any parabolic parameter $\gamma$ for $f_c(z)$ satisfies the following conditions:
    \begin{enumerate}
        \item\label{item: 4 gamma is algebraic integer} $4\gamma$ is an algebraic integer by \cref{cor: unicritical_integrality_of_parabolic_parameters}.
        \item\label{item: 2 does not divide the norm} $4\gamma$ is not in any prime ideal $\pe$ over $2$, or equivalently, the norm $N_{K/\Q}(4\gamma)$ is not divisible by $2$ by \cref{lem: Huguin's argument}.
        \item\label{item: in the disc} Every Galois conjugate of $\gamma$ is contained in the closed disc $D_0$ of radius $2$ centered at the origin since it is contained in the Mandelbrot set.
        \item\label{item: not in the main cardioid} $\gamma$ is not contained in the interior of the main cardioid $D_1$ of Mandelbrot set, where $D_1$ is defined by
        \begin{align}
            D_1 & \coloneq \{ 1/4 + r((1-\cos \theta)\exp(i\theta))\ |\ 0\leq \theta < 2\pi, 0<r<1\} \\
            &= \{\gamma \in \mathbb{C}\ |\ x + y \sqrt{-1} = 1 - 4\gamma, (x^2+y^2)(x^2+y^2-4x) - 4y^2 < 0 \}.
        \end{align}
        This is because $f_\gamma$ has an attracting fixed point if $\gamma \in D_1$.
        \item\label{item: not in the circle} $\gamma$ is not contained in the interior of the secondary large domain $D_2$ of Mandelbrot set, where $D_2$ is defined by
        \[
            D_2 \coloneq \left\{ \gamma \in \mathbb{C} \relmiddle| \abs{ \gamma + 1} < \frac{1}{4}\right\}.
        \]
        This is because $f_\gamma$ has an attracting cycle of period $2$ if $\gamma \in D_2$.
    \end{enumerate}
    For a fixed positive integer $D$, the number $\gamma \in \C$ of degree $[\Q(\gamma):\Q] \leq D$ satisfying \cref{item: 4 gamma is algebraic integer,item: in the disc} is finite.
    Thus, we can list candidates for parabolic parameters that satisfy all of these conditions.
    Using the numerically verified computation, we check two things.
    If the orbit of $0$ under $f_\gamma$ is not contained in $D_0$ up to the fixed number of iterates, the parameter $\gamma$ is not a parabolic parameter.
    If the orbit of $0$ looks to converge to some periodic cycle of period $N$, compute the exact description of $\Phi_N^*(z)$ or $\delta_N(x)$.
    If we can find a root of $\delta_N(x)$ whose absolute value is smaller than $1$, or equivalently, a root $\alpha$ of $\Phi_N^*(z)$ with $|\omega_N(\alpha)| < 1$ by the numerically verified computation, then the map $f_\gamma$ has an attracting periodic point and hence does not have a parabolic periodic point by \cref{thm: parabolic_attracting_critical}.
    For the remaining candidates $\gamma$, find integers $m,n$ with $m \mid n$ such that $\Delta_{n,m}(\gamma) = 0$. Believing the irreducibility of all $\Delta_{n,m}$, we can efficiently find such $n,m$. See \cref{subsec: Parabolic parameters under the irreducibility conjecture} for the degree estimate of $\Delta_{n,m}$.
    In principle, we can give candidates for parabolic parameters of the fixed degree over $\Q$ for $f_{d,c}(z) = z^d + c$ for every integer $ d \geq 2 $ in this way.

\begin{proof}[Proof of $ \cref{thm: parabolic_rational_parameters} $]
    Continue to refer to conditions \cref{item: 4 gamma is algebraic integer}--\cref{item: not in the circle} above.
    It is proved that the set of totally real parabolic parameters is
    \[
        \left\{ -\frac{7}{4}, -\frac{5}{4}, -\frac{3}{4}, \frac{1}{4}\right\}
    \]
    in \cite{BK22}.
    Since rational and real quadratic numbers are totally real, it is enough to determine the imaginary quadratic parameters.
    As a specific example of the determination of parabolic parameters, we also treat the case of rational parameters here.
    For a rational number $ \gamma \in \Q $, assume that the polynomial $ f_\gamma(z) = z^2 + \gamma $ has a parabolic periodic point in $ \C $, hence $\gamma$ satisfies all conditions \cref{item: 4 gamma is algebraic integer}--\cref{item: not in the circle}.
    Such rational numbers are only $-7/4, -5/4, -3/4, -1/4, 1/4$.
    For these parameters, we can directly check that there are parabolic or attracting periodic points, as shown in the following table.
    \begin{table}[htb]
  \caption{attracting/parabolic periodic point and multiplier}
  \label{table: case-by-case_unicritical}
  \centering
  \renewcommand{\arraystretch}{1.5} % change the height of boxes in the table
  \begin{tabular}{|c|ccccc|} \hline
    $\gamma$            & $-\frac{7}{4}$  & $-\frac{5}{4}$ & $-\frac{3}{4}$ & $-\frac{1}{4}$          & $\frac{1}{4}$ \\ \hline
    periodic point(s)   & $\alpha_1, \alpha_2, \alpha_3$  & $\frac{-1\pm\sqrt{2}}{2}$ & $-\frac{1}{2}$  & $\frac{1-\sqrt{2}}{2}$ & $\frac{1}{2}$ \\ \hline
    period              & $3$ & $2$ & $1$ & $1$ & $1$ \\ \hline
    multiplier          & $1$ & $-1$ & $-1$ & $1-\sqrt{2}$ & $1$           \\ \hline
  \end{tabular}
\end{table}

Next, we determine the (imaginary) quadratic parabolic parameters.
We consider $K = \Q(\sqrt{-D})$ with a squarefree positive integer $D \in \Z_{\geq 1}$.
Let
\begin{align}
    \beta = 
    \begin{cases}
        \sqrt{-D} & D\equiv 1,2 \mod 4\\
        \frac{1+\sqrt{-D}}{2} & D\equiv 3 \mod 4.
    \end{cases}
\end{align}
Then, we note that $\mathcal{O}_K = \Z[\beta]$.
Let $\gamma \in K$ be an imaginary quadratic parabolic parameter for $f_c(z) = z^2 + c$.
By \cref{item: 4 gamma is algebraic integer}, we have $4\gamma = a+ b\beta$ for some integers $a,b$ with $b\neq 0$.
Since the Galois conjugate $\overline{\gamma}$ is also a parabolic parameter for $f_c(z)$, we may assume that $b$ is positive by changing the role of $\gamma$ and $\overline{\gamma}$ if it is necessary.

When $D \equiv 1, 2 \mod 4$,
since we have $|\gamma| = \frac{\sqrt{a^2 + b^2 D}}{4} \geq \frac{b \sqrt{D}}{4} (\geq \sqrt{D}/4)$ and
    $|\gamma|<2$, we get
    \begin{align}
        D &= 1,5,13,17,21,29,33,37,41,53,57,61,\\
        &\hphantom{==} 2,6,10,14,22,26,30,34,38,42,46,58,62,
    \end{align}
    where note that $D$ is square free.

When $D \equiv 3 \mod 4$, since we have $|\gamma| = \frac{\sqrt{(2a+b)^2 + b^2 D}}{2\cdot 4} \geq b\sqrt{D}/8$ and $|\gamma| < 2$, we get
    \begin{align}
        D &= 3,7,11,15,19,23,31,35,39,43,47,51,55,59,67,71,79,83,87,91,95,\\
        &\hphantom{=} 103,107,111,115,119,123,127,131,139,143,151,155,159,163,167,179,\\
        &\hphantom{=} 183,187,191,195,199,203,211,215,219,223,227,231,235,239,247,251,255.
    \end{align}

By $|\gamma| < 2$, the possible values $(D,b)$ are as follows:
\[
    \begin{array}{|c||c|c|c|c|}\hline
        D\equiv 1\ (4) & 1 & 5 & 13 & \geq 17\\ \hline
        b & 1,2, \ldots,7 & 1,2,3 & 1,2 & 1 \\ \hline
    \end{array}
\]
\[
    \begin{array}{|c||c|c|c|c|}\hline
        D\equiv 2\ (4) & 6 & 10 & 14 & \geq 22\\ \hline
        b & 1,2, \ldots,5 & 1,2,3 & 1,2 & 1 \\ \hline
    \end{array}
\]
\[
    \begin{array}{|c||c|c|c|c|c|}\hline
        D\equiv 3\ (8) & 3 & 11 & 19 & 35,43,51,59 & \geq 67\\ \hline
        b & 1,\ldots,12 & 1,\ldots,4 & 1,2,3 & 1,2 & 1 \\ \hline
    \end{array}
\]
\[
    \begin{array}{|c||c|c|c|c|c|}\hline
        D\equiv 7\ (8) & 7 & 15 & 23 & 31,39,47,55 & \geq 71 \\ \hline
        b & 1,\ldots, 6 & 1,\ldots,4 & 1,2,3 & 1,2 & 1 \\ \hline
    \end{array}
\]

Depending on the congruence condition of $D$, the value of the norm $N_{K/\Q}(4\gamma) \mod 2$ is as follows.
\begin{align}
    N_{K/\Q}(4\gamma) \equiv
    \begin{cases}
        (a+b)^2 \equiv a+b & (D\equiv 1 \mod 4)\\
        a^2     \equiv a   & (D\equiv 2 \mod 4)\\
        a^2+ab+b^2 \equiv (a+1)(b+1)+ 1 & (D\equiv 3 \mod 8)\\
        a^2 + ab = a(a+b)   & (D\equiv 7 \mod 8).
    \end{cases}
\end{align}

By the computation with treating exact algebraic numbers, we can list the possible values of $(D, a, b)$ such that the corresponding parameter $\gamma = (a + b\beta)/4$ satisfies all conditions \cref{item: 4 gamma is algebraic integer}--\cref{item: not in the circle}.
The number of remaining parameters are $110$, $0$, $384$, and $48$ for $D \equiv 1 \mod 4$, $D \equiv 2 \mod 8$, $D \equiv 3 \mod 8$, and $D \equiv 7 \mod 8$, respectively.
%\[
%    \begin{array}{|c||c|c|c|c|}\hline
%        D & 1\mod 4 & 2\mod 4 & 3\mod 8 & 7\mod 8 \\ \hline
%        \text{number of parameters} & 110 & 0 & 384 & 48 \\ \hline
%    \end{array}
%\]

If $\gamma = (a + b\beta)/4$ is a parabolic parameter, the set $\{f_\gamma^{n}(0)\ |\ 0\leq n \leq 20\}$ must be contained in the disc $D_0$ by \cref{thm: parabolic_attracting_critical}.
By treating the parameters as an exact algebraic number with SageMath, we can see that only the parameters corresponding to
$(D,a,b) = \underline{(1,-4,1)}$,
$(1, 0, 3)$,
$\underline{(1, 1, 2)}$,
$\underline{( 3, -5, 1 )}$,
$\underline{( 3, -4, 1 )}$,
$( 3, -3, 3 )$,
$\underline{( 3, -2, 3 )}$,
$( 3, -1, 3 )$,
$\underline{( 3, 1, 1 )}$,
$( 11, -3, 1 )$,
$( 11, -1, 2 )$,
$( 35, -1, 1 )$,
$( 43, -1, 1 )$,
$( 51, -1, 1 )$
satisfy this condition, where the underlined tuples $(D,a,b)$ correspond to parabolic parameters by \cref{prop: mn_pair} and \cref{tab: Delta_mn}.
See also \cref{table: case-by-case_unicritical_imaginary_quadratic} for the explicit description of the parabolic periodic points for each parabolic parameter.

It is enough to show that the parameters $\gamma = (a + b \beta)/4$ corresponding to the tuples $(D, a, b)$ without underlining are not parabolic.
As is written before the proof, the numerically verified computation tells us that each parameter has the following properties:

\begin{table}[h]
    \begin{tabular}{|c|c|l|}\hline
        $( D, a, b )$ & $4\gamma$ & property of $f_{\gamma}$\\ \hline \hline
        $( 1, 0, 3 )$ & $3\sqrt{-1}$ & $34$-rd orbit of $0$ is outside $D_0$.\\
        $( 3, -3, 3 )$ & $(-3+3\sqrt{-3})/2$ & $149$-th orbit of $0$ is outside $D_0$.\\
        $( 3, -1, 3 )$ & $(1 + 3\sqrt{-3})/2$ & $23$-rd orbit of $0$ is outside $D_0$.\\
        $( 11, -3, 1 )$ & $(-5 + \sqrt{-11})/2$ & Has an attractor of period $7$.\\
        $( 11, -1, 2 )$ & $\sqrt{-11}$ & $22$-nd orbit of $0$ is outside $D_0$.\\
        $( 35, -1, 1 )$ & $(-1 + \sqrt{-35})/2$ & Has an attractor of period $3$.\\
        $( 43, -1, 1 )$ & $(-1 + \sqrt{-43})/2$ & Has an attractor of period $3$.\\
        $( 51, -1, 1 )$ & $(-1 + \sqrt{-51})/2$ & $26$-th orbit of $0$ is outside $D_0$.\\ \hline
    \end{tabular}
\end{table}

Here, the fact that $f_\gamma(z)$ has the attracting periodic points for $\gamma = (-1+\sqrt{-43})/8$, $(-1+\sqrt{-35})/8$, $(-5+\sqrt{-11})/8$ is checked by the computation as follows.
Note that if the period of the attracting periodic point is large, computing the multiplier polynomials and the dynatomic polynomials can be difficult.
The $3$-rd multiplier polynomial for $\gamma = (-1+\sqrt{-43})/8$ is
\[
    \delta_3(x) = x^2 + (-\sqrt{-43} - 15)x - \sqrt{-43} - 12.
\]
One of the solutions is $x \approx -0.794010151370859 - 0.0814301974776708\sqrt{-1}$ whose absolute value is smaller than $1$.
The $3$-rd multiplier polynomial for $\gamma= (-1+\sqrt{-35})/8$ is
\[
    \delta_3(x) = x^2 + (-\sqrt{-35} - 15)x + 1.
\]
One of the solutions is $x \approx 0.0577941490144990 - 0.0229713342664962\sqrt{-1}$ whose absolute value is smaller than $1$.
The $7$-th dynatomic polynomial for $\gamma = (-5+\sqrt{-11})/8$ has roots approximated to 
\begin{align}
    &-0.828984535519859 + 0.214178626239443\sqrt{-1},\\
    &-0.652947538065546 - 0.0615405901183171\sqrt{-1},\\
    &-0.628270372618086 + 0.416522134057810\sqrt{-1},\\
    &-0.473809395908544 + 0.502436943027795\sqrt{-1},\\
    &-0.403767027041217 - 0.108798933972456\sqrt{-1},\\
    &-0.202446756765942 + 0.494943652412456\sqrt{-1},\\
    &0.0163428674089548 + 0.0594765663896106\sqrt{-1}
\end{align}
that consists a $f_\gamma$-periodic orbit.
The multiplier of these roots is approximately $0.198642942872117 - 0.476329104815534\sqrt{-1}$ whose absolute value is smaller than $1$.
These computations mean that $f_\gamma$ has an attracting periodic point of the corresponding period for each $\gamma$.
\end{proof}

\begin{table}[htb]
	\caption{parabolic periodic point and multiplier for imaginary quadratic parameter}
	\label{table: case-by-case_unicritical_imaginary_quadratic}
	\centering
	\renewcommand{\arraystretch}{1.5} % change the height of boxes in the table
	\begin{tabular}{|c|cccccc|} \hline
		$\gamma$ & $ \frac{1 + 2\sqrt{-1}}{4} $ & $ \frac{-4 + \sqrt{-1}}{4} $ & $ \frac{-1 + 3\sqrt{-3}}{8} $ & $ \frac{3 + \sqrt{-3}}{8} $ & $ \frac{-7 + \sqrt{-3}}{8} $ & $ \frac{-9 + \sqrt{-3}}{8}$ \\ \hline
		periodic point(s)   & $ \frac{\iu}{2} $  & $ \frac{-1 \pm \sqrt{1 - \iu}}{2} $ & $ \frac{-1 + \sqrt{-3}}{4} $ & $ \frac{1 + \sqrt{-3}}{4} $ & $ \frac{-2 \pm (-\sqrt{3} + \iu)}{4} $ & $ \frac{-2 \pm \sqrt{6 -2 \sqrt{-3}}}{4} $ \\ \hline
		period & $1$ & $2$ & $1$ & $1$ & $2$ & $ 2 $ \\ \hline
		multiplier & $ \iu $ & $ \iu $ & $ \frac{-1 + \sqrt{-3}}{2} $ & $ \frac{1 + \sqrt{-3}}{2} $ & $ \frac{-1 + \sqrt{-3}}{2} $ & $ \frac{-1 + \sqrt{-3}}{2} $ \\ \hline
	\end{tabular}
\end{table}

% --------------------------------------------------------------------------

\subsection{Parabolic parameters under the irreducibility conjecture}
\label{subsec: Parabolic parameters under the irreducibility conjecture}
% --------------------------------------------------------------------------
In this section, we determine the parabolic parameters of $f_c(z)= z^2 + c$ of small degrees under \cref{conj: irreducibility}.

From the results in \cref{subsec: deg_estimation}, we get the tables written in \cref{subsec: lists_on_deg_of_Delta}. The results in this section follow from these tables.
Rational and quadratic parabolic parameters are unconditionally determined in \cref{thm: parabolic_rational_parameters}.

\begin{prop}\label{prop: cubic}
    If \cref{conj: irreducibility} is true, then each cubic parabolic parameter for $f_c(z) = z^2 + c$ is a root of one of the following  polynomials:
    \begin{align}
        64c^{3}& + 128c^{2} + 72c + 81,\\
        64c^{3}& + 144c^{2} + 108c + 135.
    \end{align}
\end{prop}

\begin{rem}
    Both of these two polynomials have a real root and two imaginary roots.
\end{rem}

\begin{prop}\label{prop: quartic}
    If \cref{conj: irreducibility} is true, then each quartic parabolic parameter for $f_c(z) = z^2 + c$ is a root of one of the following  polynomials: 
    \begin{align}
        256c^{4}& + 64c^{3} + 16c^{2} - 36c + 31,\\
        256c^{4}& + 32c^{2} - 64c + 17,\\
        256c^{4}& - 192c^{3} + 144c^{2} - 68c + 11,\\
        256c^{4}& + 128c^{3} - 16c^{2} - 56c + 13,\\
        256c^{4}& + 1088c^{3} + 1744c^{2} + 1252c + 341,\\
        256c^{4}& + 1024c^{3} + 1536c^{2} + 1024c + 257,\\
        256c^{4}& + 960c^{3} + 1360c^{2} + 860c + 205,\\
        256c^{4}& + 1024c^{3} + 1520c^{2} + 992c + 241.
    \end{align}
\end{prop}

\begin{rem}
    None of these polynomials have real roots.
\end{rem}

\begin{prop}\label{prop: quintic}
    If \cref{conj: irreducibility} is true, then there is no parabolic parameter of degree $ 5 $ and $ 7 $.
\end{prop}

% --------------------------------------------------------------------------

\section{Remarks on other polynomial families} \label{sec: other_families}

% --------------------------------------------------------------------------

In this section, we give some remarks on other polynomial families.

% --------------------------------------------------------------------------

\subsection{The family $ (z - c)z^d + c$} \label{subsec: (z-c)z^d+c}

% --------------------------------------------------------------------------

The method in \cref{sec: non_unicritical} works for the polynomial family $ (\tilde{f}_c(z) = (z - c)z^d + c)_c$.
We use the same notation in \cref{sec: non_unicritical} to see this. 
In addition, we define the dynatomic polynomial for $ \tilde{f}_c $ and the multiplier polynomial for $ \tilde{f}_c $ by
\begin{align}
    \tilde{\Phi}_m^{*}(z) &:= \prod_{k | m} \left( \tilde{f}_c^{\circ k}(z) - z \right)^{\mu(m/k)},\\
    \tilde{\delta}_{m}(x)^m &:= \mathrm{Res}_z \left(\tilde{\Phi}_m^{\ast}(z), x - (\tilde{f}_c^{\circ m})'(z) \right).
\end{align}
Then, the following theorem holds.

\begin{thm} \label{thm: integrality_for_f_tilde}
  For a positive integer $ m $, the polynomial $ d^{(d-1)\varepsilon_m + \deg_z \tilde{\Phi}_m / (d+1) } \tilde{\delta}_{m}(x) $ is an element of $\Z[(dc)^d, x]$ and is monic in $ (dc)^d $ up to multiplication by $ \pm 1 $.
\end{thm}

The quantity $ \varepsilon_m $ in \cref{thm: integrality_for_f_tilde} is defined as in \cref{{sec: non_unicritical}}. Namely, 
\begin{align}
    \varepsilon_m \coloneqq \sum_{k | n} \mu(n/k)  =
	\begin{cases}
		1 & \text{ if } m=1, \\
		0 & \text{ if } m \neq 1.
	\end{cases}
\end{align}

For the family $ \left(\tilde{f}_c(z) = (z - c)z^d + c\right)_c $, the following diagram commutes
\begin{equation}
    \begin{tikzcd}
		\bbP^1 \arrow[r, "\nu"] \arrow[d, "\tilde{f}_c" '] & \bbP^1 \arrow[d, "\tilde{f}_{c \cdot \zeta_d }"] \\
		\bbP^1 \arrow[r, "\nu" '] & \bbP^1
    \end{tikzcd}
\end{equation}
for $ \nu \colon \bbP^1 \longrightarrow \bbP^1 ; z \mapsto \zeta_d z$, where $ \zeta_d $ is a primitive $ d $-th root of unity.
Thus the discussion in the proof of \cref{lem: unicritical_integrality_of_R} \cref{item: lem: unicritical_invariance_of_coefficients} shows that $ \mathrm{Res}_z(\tilde{f}_c^{\circ k}(z) - z, x - (\tilde{f}_c^{\circ m})'(z)) \in \Z[c^d,x]$.
So, it is enough to show that $ d^{\deg_z \tilde{\Phi}_m / (d+1) } \tilde{\delta}_{m}(x) $ is an element of $ \Z[dc, x] $ and monic in $ dc $ up to multiplication by $ \pm 1 $.

By \cref{lem: non_unicritical_tilde_f} \cref{item: lem: non_unicritical_tilde_f: 1} and \cref{item: lem: non_unicritical_tilde_f: 2}, we have $ \tilde{f}_c^{\circ k}(z) - z = (z - c)\left( (F_k(z) + 1)^d - 1 \right)$ and 
\[ (\tilde{f}_c^{\circ m})'(z) = (F_m(z) + 1)^{d-1} (d+1)^m \calF_m \left( -\frac{dc}{d+1}, z \right).\]
We define a polynomial $H_k(z)$ by $ H_k(z) = (F_k(z) + 1)^d - 1$. Then, we have 
\begin{align}
    \tilde{\delta}_{m}(x)^m
    &= \mathrm{Res}_z (\tilde{\Phi}_m^{*}(z), x - (\tilde{f}_c^{\circ m})'(z) )\\
    &= \prod_{k | m} \mathrm{Res}_z \left( (z - c) H_k(z) , x -  (F_m(z) + 1)^{d-1} (d+1)^m \calF_m \left( -\frac{dc}{d+1}, z \right) \right)^{\mu(m/k)} \\
    &= \prod_{k | m} \left( (x - c^{md}) \mathrm{Res}_z \left( H_k(z) , x -  (F_m(z) + 1)^{d-1} (d+1)^m \calF_m \left( -\frac{dc}{d+1}, z \right) \right) \right)^{\mu(m/k)} \\
    &= (x - c^{md})^{\varepsilon_m} \prod_{k | m} \Res_z \left( H_k(z) , x -  (F_m(z) + 1)^{d-1} (d+1)^m \calF_m \left( -\frac{dc}{d+1}, z \right) \right)^{\mu(m/k)} .
\end{align}
Putting $ \tilde{R}_{k,m}(x) = \Res_z \left( H_k(z) , x -  (F_m(z) + 1)^{d-1} (d+1)^m \calF_m \left( -\frac{dc}{d+1}, z \right) \right)$, the statement \cref{thm: integrality_for_f_tilde} is reduced to the following lemma.

\begin{lem} \label{lem: f_tilde}
  The following statements hold for positive integers $k$ and $m$, which satisfy $ k|m $.
  \begin{enumerate}
      \item \label{item: integrality_f_tilde}
        $ d^{m((d+1)^{k-1}-1)}\tilde{R}_{k,m}(x) \in \Z[dc, x] $.
      \item \label{item: leading_coefficient_f_tilde}
        The leading coefficient of $ \tilde{R}_{k,m}(x) $ in $ c $ is $ \pm d^{md(d+1)^{k-1}}$.
  \end{enumerate}
\end{lem}

\begin{proof}
    Since $ F_m(z) + 1 = \prod_{i=0}^{m-1} \tilde{f}_c^{\circ i}(z) $, $ H_k(z) = \left(\prod_{i=0}^{k-1} \tilde{f}_c^{\circ i}(z) \right)^d - 1 $, and $ Z(H_k) \subset Z(\tilde{f}_c^{\circ k}(z) - z)$,
    the element $ \xi_{\alpha} := (F_k(\alpha) + 1)^{d-1} $ is a $d$-th root of unity for each $ \alpha \in Z(H_k) $. 
    Therefore, we have the expression
    \begin{align}
        & \Res_z \left( H_k(z), x -  (F_m(z) + 1)^{d-1} (d+1)^m \calF_m \left( y, z \right) \right)\\
        =& \prod_{\alpha \in Z(H_k)} \left( x -  \xi_{\alpha} (d+1)^m \calF_m \left( y, \alpha \right) \right).
    \end{align}
    Therefore, the same argument as in \cref{prop: non_unicritical_degree} shows that
    \[ \deg_c \Res_z \left( H_k(z) , x -  (F_m(z) + 1)^{d-1} (d+1)^m \calF_m \left( y, z \right) \right) = m((d+1)^{k-1}-1).\]
    This assertion proves (i).
    Next, we prove the assertion (ii).
    By using the same discussion as in \cref{prop: non_unicritical_valuation_F_k}, we get the equality
    \[ v \left( (d+1)^m \calF_m \left( -\frac{d}{d+1} c, \alpha \right) \right) = -m \] 
    for each $ \alpha \in Z(H_k)$. Therefore, we can prove
    \[ \deg_c \tilde{R}_{k,m}(x) = m((d+1)^{k}-1) \]
    and
    \[ \LT_c\left(\tilde{R}_{k,m}(x)\right) = \pm \LT_c \left( (d+1)^{\deg_z H_k} H_k\left( \frac{dc}{d+1} \right) \right)^m = \pm d^{md(d+1)^{k-1}} c^{m((d+1)^{k}-1)}.
    \]
\end{proof}

% --------------------------------------------------------------------------

\subsection{The family $ z^{d+2} + cz^2 $ } \label{subsec: z^{d+2}+cz^2}

% --------------------------------------------------------------------------

Throughout this subsection, we put $ f_c(z) \coloneq z^{d+2} + cz^2$ for an integer $ d \ge 3 $.
In this subsection, we prove \cref{thm: other_family}.
More precisely, we prove the following proposition and then observe that the same claim as in \cref{thm: unicritical_Silverman_Conj} does not hold in this situation.

\begin{prop}\label{prop: another_family}
    The first multiplier polynomial for $ f_c(z) = z^{d+2} + cz^2 $ is given by the formula
    \[ \delta_1(x) = x \left( (x-(d+2))^{d+1} + c (-cd)^d (x-2)  \right).\]
    In particular, for each integer $ n > 1 $, the following statement holds.
    \[ \LT_c(\Delta_{n,1}(c)) = \pm d^{d\varphi(n)} \cyc_n(2) c^{(d+1)\varphi(n)} .\]
\end{prop}

Although we give the computation only for $ m=1 $, a better observation might be obtained by computing the $m$-th multiplier polynomial in general.

\begin{proof}[Proof of \cref{prop: another_family}]
    A direct computation shows that if $ \alpha $ is a fixed point of $ f_c(z) = z^{d+2} + cz^2 $, then $ \alpha = 0 $ or $ \alpha \in Z(z^{d+1}+cz-1) $.
    By the definition of the multiplier, we obtain $ \omega_1(0) = 0 $ and $ \omega_1(\alpha) = (d+2) \alpha^{d+1} + 2c\alpha = (d+2) - cd\alpha $ for $\alpha \in Z(z^{d+1}+cz-1) $.
    Hence, we can explicitly calculate $\delta_1(x)$ as follows.
    \begin{align}
    \delta_1(x) &= \prod_{f_c(\alpha)=\alpha} (x - \omega_1(\alpha))\\
    &= x \prod_{\alpha \in Z(z^{d+1}+cz-1)} (x - (d+2) + cd\alpha) \\
    &= x (-cd)^{d+1} \prod_{\alpha \in Z(z^{d+1}+cz-1)} \left(\frac{x - (d+2)}{-cd} - \alpha\right) \\
    &= x (-cd)^{d+1} \left( \left(\frac{x - (d+2)}{-cd}\right)^{d+1} + c\frac{x - (d+2)}{-cd} -1 \right) \\
    &= x \left( (x-(d+2))^{d+1} + c (-cd)^d (x-2)  \right).
    \end{align}
    
    By the explicit calculation of the multiplier polynomial $ \delta_1(x) $, we can easily calculate the leading term of $ \Delta_{n,1}(c) $ as follows.
    
    \begin{align}
    \LT_c(\Delta_{n,1}(c)) &= \LT_c\left( \prod_{\Phi_n^{\mathrm{cyc}}(\zeta) = 0} \delta_1(\zeta) \right) \\
    &= \prod_{\Phi_n^{\mathrm{cyc}}(\zeta) = 0} \zeta c (-cd)^d (\zeta - 2)\\
    &= \pm d^{d\varphi(n)} \Phi_n^{\mathrm{cyc}}(2) c^{(d+1)\varphi(n)}.
    \end{align}
\end{proof}

Bang's theorem \cite{Bang86} on cyclotomic numbers shows that $ \Phi_n^{\mathrm{cyc}}(2) $ has a prime divisor $ q $ such that $ q \equiv 1 \text{ mod } n $ except for $ \Phi_6^{\mathrm{cyc}}(2) = 3 $. 
Therefore, the uniform integrality cannot be guaranteed as in other examples.

% --------------------------------------------------------------------------

\appendix
\def\thesection{\Alph{section}}

%\section*{Appendix}
\addcontentsline{toc}{section}{Appendix}

% --------------------------------------------------------------------------

% --------------------------------------------------------------------------

\section{Equalities on dynatomic modular curves} \label{sec: equality}

% --------------------------------------------------------------------------

The first equality in \cref{prop: equality_main} follows from the following lemma.

\begin{lem} \label{lem: equality_dyn_poly}
	For coprime integers $ l $ and $ n $, we have
	\[
	\prod_{d \mid l} \Phi_{f, dn}^\ast (z)
	= \Phi_{f^{\circ l}, n}^\ast (z).
	\]
\end{lem}

\begin{proof}
	For a divisor $ d $ of $ l $, we have
	\[
	\Phi_{f, dn}^\ast (z)
	= \prod_{e \mid d} \prod_{e' \mid n} \left( f^{\circ ee'}(z) - z \right)^{\mu(d/e) \mu(n/e')} 
	= \prod_{e \mid d} \Phi_{f^{\circ e}, n}^\ast (z)^{\mu(d/e)}
	\]
	since $ d $ and $ n $ are coprime.
	Thus, we have
	\[
	\prod_{d \mid l} \Phi_{f, dn}^\ast (z)
	= \prod_{d \mid l} \prod_{e \mid d} \Phi_{f^{\circ e}, n}^\ast (z)^{\mu(d/e)}.
	\]
	By letting $ e' \coloneqq d/e $, this is equal to
	\[
	\prod_{e \mid l} \prod_{e' \mid l/e} \Phi_{f^{\circ e}, n}^* (z)^{\mu(e')}
	= \prod_{e \mid l} \Phi_{f^{\circ e}, n}^\ast (z)^{\sum_{e' \mid l/e} \mu(e')}
	= \Phi_{f^{\circ l}, n}^* (z).
	\]
\end{proof}

\begin{proof}[Proof of $ \cref{prop: equality_main} $]
	It suffices to show the second equality.
	Since $m' > 1$, we have
	\[
	\Phi_{f^{\circ \widetilde{m}}, m'}^\ast (z)
	= \prod_{d \mid m'} \left( f^{\circ \widetilde{m}d}(z) - z \right)^{\mu(m'/d)}
	= \prod_{d \mid m'} \left( \frac{f^{\circ \widetilde{m}d}(z) - z}{z - \alpha} \right)^{\mu(m'/d)}.
	\]
	By taking a limit as $ z \to \alpha $, we obtain
	\[
	\Phi_{f^{\circ \widetilde{m}}, m'}^\ast (\alpha)
	= \prod_{d \mid m'} \left( \left( f^{\circ \widetilde{m}d} \right)' (\alpha) - 1 \right)^{\mu(m'/d)}
	= \prod_{d \mid m'} \left( \lambda^{\widetilde{m}d/k} - 1 \right)^{\mu(m'/d)}
	= \cyc_{m'} \left( \lambda^{\widetilde{m}/k} \right).
	\]
\end{proof}

% --------------------------------------------------------------------------
\newpage
\section{Lists of $\delta_m(x)$ and $\Delta_{n,m}$ for our polynomials}
\label{sec: list_of_delta_and_Delta}

% --------------------------------------------------------------------------
This section shows computational results on $\delta_{m}$ and $\Delta_{n,m}$.
The following calculations were performed using SageMath\cite{sagemath}.
Let $f_c(z) = z^d +c$.
Then, for integers $d\geq 2$ and $m\geq 1$, the polynomial $\Psi(x, C)$ such that $\delta_m(x) = \Psi(x,d^d c^{d-1})$ are explicitly written as follows.

\begin{table}[h]
    \begin{tabular}{|c|c|c|p{14cm}|}
    \hline
    $d$ & $m$ & $\deg_z \Phi_m^*$ &  $\Psi(x, C)$ such that $\delta_m(x) = \Psi(x,d^d c^{d-1})$ for $f_c(z) = z^d +c$.\\ \hline \hline
    $2$ & $1$ & $2$ & $C + x^{2} - 2 x$\\ \hline
    $2$ & $2$ & $2$ & $- C + x - 4$\\ \hline
    $2$ & $3$ & $6$ & $C^{3} + 8 C^{2} - 2 C x + 16 C + x^{2} - 16 x + 64$\\ \hline \hline
    $3$ & $1$ & $3$ & $- C + x^{3} - 6 x^{2} + 9 x$\\ \hline
    $3$ & $2$ & $6$ & $- C^{2} + 6 C x - 54 C + x^{3} - 27 x^{2} + 243 x - 729$\\ \hline
    $3$ & $3$ & $24$ & $C^{8} + C^{7} x + 162 C^{7} + C^{6} x^{2} + 108 C^{6} x + 10935 C^{6} - 2 C^{5} x^{3} + 54 C^{5} x^{2} + 2187 C^{5} x + 393660 C^{5} - 2 C^{4} x^{4} + 162 C^{4} x^{3} - 1458 C^{4} x^{2} - 196830 C^{4} x + 9034497 C^{4} - 2 C^{3} x^{5} - 324 C^{3} x^{4} + 23328 C^{3} x^{3} - 157464 C^{3} x^{2} - 11691702 C^{3} x + 172186884 C^{3} + C^{2} x^{6} + 162 C^{2} x^{5} - 19683 C^{2} x^{4} + 551124 C^{2} x^{3} + 1594323 C^{2} x^{2} - 258280326 C^{2} x + 2711943423 C^{2} + C x^{7} - 108 C x^{6} + 2187 C x^{5} + 196830 C x^{4} - 13286025 C x^{3} + 344373768 C x^{2} - 4261625379 C x + 20920706406 C + x^{8} - 216 x^{7} + 20412 x^{6} - 1102248 x^{5} + 37200870 x^{4} - 803538792 x^{3} + 10847773692 x^{2} - 83682825624 x + 282429536481$ \\ \hline \hline
    $4$ & $1$ & $4$ & $C + x^{4} - 12 x^{3} + 48 x^{2} - 64 x$ \\ \hline
    $4$ & $2$ & $12$ & $C^{3} - C^{2} x^{2} - 48 C^{2} x + 768 C^{2} - C x^{4} + 1536 C x^{2} - 32768 C x + 196608 C + x^{6} - 96 x^{5} + 3840 x^{4} - 81920 x^{3} + 983040 x^{2} - 6291456 x + 16777216$
        \\ \hline\hline
    $5$ & $1$ & $5$ & $- C + x^{5} - 20 x^{4} + 150 x^{3} - 500 x^{2} + 625 x$\\ \hline
    $5$ & $2$ & $20$ & $C^{4} - 20 C^{3} x^{2} - 500 C^{3} x + 12500 C^{3} - 2 C^{2} x^{5} - 150 C^{2} x^{4} + 9375 C^{2} x^{3} + 31250 C^{2} x^{2} - 5859375 C^{2} x + 58593750 C^{2} + 20 C x^{7} - 2500 C x^{6} + 112500 C x^{5} - 1562500 C x^{4} - 39062500 C x^{3} + 1757812500 C x^{2} - 24414062500 C x + 122070312500 C + x^{10} - 250 x^{9} + 28125 x^{8} - 1875000 x^{7} + 82031250 x^{6} - 2460937500 x^{5} + 51269531250 x^{4} - 732421875000 x^{3} + 6866455078125 x^{2} - 38146972656250 x + 95367431640625$
%    \\ \hline 
%    $5$ & $3$ & $120$ & Too long.
    \\ \hline
%    \hline
%    $6$ & $1$ & $6$ & $C + x^{6} - 30 x^{5} + 360 x^{4} - 2160 x^{3} + 6480 x^{2} - 7776 x$\\ \hline
%    $6$ & $2$ & $30$ & $- C^{5} + C^{4} x^{3} + 360 C^{4} x^{2} + 6480 C^{4} x - 233280 C^{4} + 2 C^{3} x^{6} + 900 C^{3} x^{5} - 2799360 C^{3} x^{3} + 33592320 C^{3} x^{2} + 1451188224 C^{3} x - 21767823360 C^{3} - 2 C^{2} x^{9} - 720 C^{2} x^{8} + 116640 C^{2} x^{7} - 3499200 C^{2} x^{6} - 209952000 C^{2} x^{5} + 15237476352 C^{2} x^{4} - 261213880320 C^{2} x^{3} - 2742745743360 C^{2} x^{2} + 126949945835520 C^{2} x - 1015599566684160 C^{2} - C x^{12} + 64800 C x^{10} - 13063680 C x^{9} + 1284906240 C x^{8} - 75461787648 C x^{7} + 2742745743360 C x^{6} - 56422198149120 C x^{5} + 296216540282880 C x^{4} + 16249593066946560 C x^{3} - 460675963447934976 C x^{2} + 5264868153690685440 C x - 23691906691608084480 C + x^{15} - 540 x^{14} + 136080 x^{13} - 21228480 x^{12} + 2292675840 x^{11} - 181579926528 x^{10} + 10894795591680 x^{9} - 504273395957760 x^{8} + 18153842254479360 x^{7} - 508307583125422080 x^{6} + 10979443795509116928 x^{5} - 179663625744694640640 x^{4} + 2155963508936335687680 x^{3} - 17911081458855711866880 x^{2} + 92114133216972232458240 x - 221073919720733357899776$\\ \hline
%    $6$ & $3$ & $210$ & Too long. \\ \hline
    \end{tabular}
\end{table}

Let $f_c(z) = z^{d+1} + cz$.
Then, for integers $d\geq 1$ and $m\geq 1$, the polynomials $\Psi(x, C)$ such that $d^{(deg_z \Phi_m^*)/(d+1)}\delta_m(x) = \Psi(x,dc)$ are explicitly written as follows.
\begin{table}[h]
    \begin{tabular}{|c|c|c|p{14cm}|}
    \hline
    $d$ & $m$ & $\deg_z \Phi_m^\ast$ &  $\Psi(x, C)$ such that $d^{(deg_z \Phi_m^*)/(d+1)}\delta_m(x) = \Psi(x,dc)$ for $f_c(z) = z^{d+1}+cz$\\ \hline \hline
    $2$ & $1$ & $3$ & $\left(- C + 2 x\right) \left(C + x - 3\right)^{2}$ \\ \hline
    $2$ & $2$ & $6$ & $\left(C^{2} + 2 x - 18\right)^{2} \left(- C^{2} - 6 C + x - 9\right)$ \\ \hline
    $2$ & $3$ & $12$ &$C^{12} + 3 C^{11} - 72 C^{10} + C^{9} x - 216 C^{9} + 1539 C^{8} - 99 C^{7} x + 4617 C^{7} - 12 C^{6} x^{2} - 108 C^{6} x - 5832 C^{6} - 36 C^{5} x^{2} + 2916 C^{5} x - 17496 C^{5} + 396 C^{4} x^{2} + 5832 C^{4} x - 131220 C^{4} + 4 C^{3} x^{3} + 972 C^{3} x^{2} - 14580 C^{3} x - 393660 C^{3} + 48 C^{2} x^{3} - 1296 C^{2} x^{2} - 34992 C^{2} x + 944784 C^{2} - 144 C x^{3} + 11664 C x^{2} - 314928 C x + 2834352 C + 16 x^{4} - 1728 x^{3} + 69984 x^{2} - 1259712 x + 8503056$\\ \hline\hline
    $3$ & $1$ & $4$ & $\left(- C + 3 x\right) \left(C + x - 4\right)^{3}$ \\ \hline
    $3$ & $2$ & $12$ & $\left(- C^{4} - 4 C^{3} - 2 C^{2} x - 16 C^{2} - 12 C x + 192 C + 3 x^{2} - 96 x + 768\right)^{3}$ \\ \hline\hline
    $4$ & $1$ & $5$ & $\left(- C + 4 x\right) \left(C + x - 5\right)^{4}$ \\ \hline
    $4$ & $2$ & $20$ & $\left(- C^{2} - 10 C + x - 25\right)^{2} \left(- C^{4} - 3 C^{2} x - 25 C^{2} + 4 x^{2} - 200 x + 2500\right)^{4}$ \\ \hline\hline
    $5$ & $1$ & $6$ & $(- C + 5 x) (C + x - 6)^{5}$ \\ \hline
    $5$ & $2$ & $30$ & $(C^{6} + 6 C^{5} + 3 C^{4} x + 36 C^{4} + 24 C^{3} x + 216 C^{3} - 9 C^{2} x^{2} + 288 C^{2} x + 1296 C^{2} - 30 C x^{2} + 2160 C x - 38880 C + 5 x^{3} - 540 x^{2} + 19440 x - 233280)^{5}$ \\ \hline
%    \hline
%    $6$ & $1$ & $7$ & $\left(- C + 6 x\right) \left(C + x - 7\right)^{6}$ \\ \hline
    \end{tabular}
\end{table}

\newpage
Let $f_c(z) = (z-c)z^d + c$.
Then, for integers $d\geq 1$ and $m\geq 1$, there is a polynomial $\Psi(x, C)$ such that $d^{(d-1)\varepsilon_m + (deg_z \Phi_m^*)/(d+1)}\delta_m(x) = \Psi(x,(dc)^d)$, where $\varepsilon_m$ is $1$ when $m=1$, and is $0$ otherwise.
The polynomials $\Psi(x, C)$ for some pairs $(d,m)$ are explicitly written as follows.
\begin{table}[h]
    \begin{tabular}{|c|c|c|p{14cm}|}\hline
    $d$ & $m$ & $\deg_z \Phi_m^\ast$ &  $\Psi(x, C)$ such that $d^{(d-1)\varepsilon_m + (deg_z \Phi_m^*)/(d+1)}\delta_m(x) = \Psi(x,(dc)^d)$ for $f_c(z) = (z-c)z^d + c$\\
    \hline\hline
    $2$ & $1$ & $3$ & $\left(- C + 4 x\right) \left(- C + x^{2} - 6 x + 9\right)$ \\ \hline
    $2$ & $2$ & $6$ & $\left(C + 2 x - 18\right) \left(- C^{2} + C x + 27 C + 2 x^{2} - 36 x + 162\right)$ \\ \hline \hline
    $3$ & $1$ & $4$ & $\left(- C + 27 x\right) \left(C + x^{3} - 12 x^{2} + 48 x - 64\right)$ \\ \hline
    $3$ & $2$ & $12$ & $- C^{4} - 12 C^{3} x + 704 C^{3} - 26 C^{2} x^{3} + 240 C^{2} x^{2} + 5376 C^{2} x - 151552 C^{2} - 324 C x^{4} + 13824 C x^{3} - 165888 C x^{2} + 7077888 C + 27 x^{6} - 2592 x^{5} + 103680 x^{4} - 2211840 x^{3} + 26542080 x^{2} - 169869312 x + 452984832$ \\ \hline \hline
    $4$ & $1$ & $5$ & $\left(- C + 256 x\right) \left(- C + x^{4} - 20 x^{3} + 150 x^{2} - 500 x + 625\right)$ \\ \hline
    $4$ & $2$ & $20$ & $(C^{2} - 17 C x^{2} + 250 C x - 5625 C + 16 x^{4} - 1600 x^{3} + 60000 x^{2} - 1000000 x + 6250000)(C^{3} + 18 C^{2} x^{2} - 100 C^{2} x - 6250 C^{2} + 33 C x^{4} + 500 C x^{3} - 71250 C x^{2} + 562500 C x + 9765625 C + 16 x^{6} - 2400 x^{5} + 150000 x^{4} - 5000000 x^{3} + 93750000 x^{2} - 937500000 x + 3906250000)$ \\ \hline
    \end{tabular}
\end{table}

Let $f_c(z) = z^{d+2} + cz^2$. Related to $\Delta_{n,m}$, the values $\Res_x(\cyc_n, \delta_m)$ for integers $d,n,m \geq 1$ are explicitly computed as follows.
One can see that they are polynomials in $c^{d+1}$ as the other polynomials we studied in this paper.
Can the readers find out some patterns in the leading coefficients?

\begin{table}[h]
    \begin{tabular}{|c|c|c|c|p{11.7cm}|}
    \hline
    $d$ & $n$ & $m$ & Leading coefficient & $\Res_x(\cyc_n, \delta_m)$ for $f_c(z) = z^{d+2} + cz^2$\\ \hline \hline
    $2$ & $1$ & $1$ & $-2^2$ & $- 4 c^{3} - 27$ \\ \hline
$2$ & $1$ & $2$ & $2^{10}\cdot 3$ & $(12 c^{3} + 125)(256 c^{9} + 3328 c^{6} - 19872 c^{3} + 91125)$ \\ \hline
$2$ & $2$ & $1$ & $2^2 \cdot 3$ & $12 c^{3} + 125$ \\ \hline
$2$ & $2$ & $2$ & $2^{10}\cdot 5$ & $5120 c^{12} + 125184 c^{9} + 340608 c^{6} - 2967452 c^{3} + 24137569$ \\ \hline
$2$ & $3$ & $1$ & $2^4 \cdot 7$ & $112 c^{6} + 1980 c^{3} + 9261$ \\ \hline
$2$ & $3$ & $2$ & $2^{20}\cdot 3\cdot 7$ & $22020096 c^{24} + 1068761088 c^{21} + 15903686656 c^{18} + 47302946816 c^{15} - 293551023104 c^{12} + 3011458104576 c^{9} + 19218341274768 c^{6} - 104174224739676 c^{3} + 413976684737889$ \\ \hline \hline
$3$ & $1$ & $1$ & $-3^3$ & $- 27 c^{4} - 256$ \\ \hline
$3$ & $1$ & $2$ & $2^{12} \cdot 3^{23}$ & $14281868906496 (c^{4} + 16)^{4}(27 c^{4} + 256)$ \\ \hline
$3$ & $2$ & $1$ & $3^4$ & $81(c^{4} + 16)$ \\ \hline
$3$ & $2$ & $2$ & $2^{12}\cdot 3^8 \cdot 5$ & $4096(9 c^{4} + 169)^{3}(3645 c^{8} + 77328 c^{4} + 456976)$ \\ \hline
$3$ & $3$ & $1$ & $3^6 \cdot 7$ & $5103 c^{8} + 131517 c^{4} + 923521$ \\ \hline
$3$ & $3$ & $2$ & $2^4\cdot 3^{31} \cdot 7$ & $3486784401(27 c^{8} + 495 c^{4} + 2401)(144 c^{8} + 5196 c^{4} + 47089)^{3}(5103 c^{8} + 131517 c^{4} + 923521)$ \\ \hline \hline
$4$ & $1$ & $1$ & $-2^8$ & $- 256 c^{5} - 3125$ \\ \hline
$4$ & $1$ & $2$ & $2^{48}\cdot 3$ & $(768 c^{5} + 16807)(1099511627776 c^{25} + 157582350090240 c^{20} + 7170620588556288 c^{15} + 81096927846400000 c^{10} - 943965155000000000 c^{5} + 8620460479736328125)$ \\ \hline
$4$ & $2$ & $1$ & $2^8 \cdot 3$ & $768 c^{5} + 16807$ \\ \hline
$4$ & $2$ & $2$ & $2^{48} \cdot 5$ & $1407374883553280 c^{30} + 237348276553121792 c^{25} + 14256316535097262080 c^{20} + 331738425127665139712 c^{15} + 1211789484758147465216 c^{10} - 21435361197289473187072 c^{5} + 333446267951815307088493$ \\ \hline
$4$ & $3$ & $1$ & $2^{16}\cdot 7$ & $458752 c^{10} + 15544576 c^{5} + 147008443$ \\ \hline
$4$ & $3$ & $2$ & $2^{96}\cdot 3\cdot 7$ & Too long %$1663791412799551089464422957056 c^{60} + 558847295239726894983605070594048 c^{55} + 80343422978545178473023798230646784 c^{50} + 6387379474476181425186833349545558016 c^{45} + 300973132443226201710377437780424785920 c^{40} + 8238599596994893636365612074186590650368 c^{35} + 113029414223055564770058427623153902551040 c^{30} + 340965612106749985891955947570558148804608 c^{25} - 2529612318493115497291975993762392049713152 c^{20} + 121226066679057467315391896111729561776947200 c^{15} + 949913164660832076733462655861531893638496256 c^{10} - 9916131664638060373915397287557016201921286912 c^{5} + 74550788499974862015474844268663989161166708957$
\\ \hline
    \end{tabular}
\end{table}

%---------------------------------------------------------------------------
\newpage
\section{Tabless on the degrees of $\Delta_{n,m}$ for $f_c(z)=z^2+c$}
\label{subsec: lists_on_deg_of_Delta}

%---------------------------------------------------------------------------
In this section, we give some tables that we used in \cref{subsec: Parabolic parameters under the irreducibility conjecture}
\begin{table}[b]
    \centering
    \begin{tabular}{|c|rrr|}\hline
        $n$ & $\nu(n)$ & $\varphi(n)$ & $d(n)$ \\ \hline
        1 & 1 & 1 & 1 \\
        2 & 1 & 1 & 0 \\
        3 & 3 & 2 & 1 \\
        4 & 6 & 2 & 3 \\
        5 & 15 & 4 & 11 \\
        6 & 27 & 2 & 20 \\
        7 & 63 & 6 & 57 \\
        8 & 120 & 4 & 108 \\
        9 & 252 & 6 & 240 \\
        10 & 495 & 4 & 472 \\
        11 & 1023 & 10 & 1013 \\
        12 & 2010 & 4 & 1959 \\
        13 & 4095 & 12 & 4083 \\
        14 & 8127 & 6 & 8052 \\
        15 & 16365 & 8 & 16315 \\
        16 & 32640 & 8 & 32496 \\
        17 & 65535 & 16 & 65519 \\
        18 & 130788 & 6 & 130464 \\
        19 & 262143 & 18 & 262125 \\
        20 & 523770 & 8 & 523209 \\
        21 & 1048509 & 12 & 1048353 \\
        22 & 2096127 & 10 & 2095084 \\
        23 & 4194303 & 22 & 4194281 \\
        24 & 8386440 & 8 & 8384100 \\
        25 & 16777200 & 20 & 16777120 \\
        26 & 33550335 & 12 & 33546216 \\
        27 & 67108608 & 18 & 67108068 \\
        28 & 134209530 & 12 & 134201223 \\
        29 & 268435455 & 28 & 268435427 \\
        30 & 536854005 & 8 & 536836484 \\
        31 & 1073741823 & 30 & 1073741793 \\
        32 & 2147450880 & 16 & 2147417952 \\
        33 & 4294966269 & 20 & 4294964173 \\
        34 & 8589869055 & 16 & 8589803488 \\
        35 & 17179869105 & 24 & 17179868739 \\
        36 & 34359605280 & 12 & 34359469848 \\
        37 & 68719476735 & 36 & 68719476699 \\
        38 & 137438691327 & 18 & 137438429148 \\
        39 & 274877902845 & 24 & 274877894595 \\
        40 & 549755289480 & 16 & 549754764132 \\
        41 & 1099511627775 & 40 & 1099511627735 \\
        42 & 2199022198821 & 12 & 2199021133728 \\
        43 & 4398046511103 & 42 & 4398046511061 \\
        44 & 8796090925050 & 20 & 8796088826787 \\
        45 & 17592186027780 & 24 & 17592185993904 \\
        46 & 35184367894527 & 22 & 35184363700180 \\
        47 & 70368744177663 & 46 & 70368744177617 \\
        48 & 140737479934080 & 16 & 140737471477920 \\
        49 & 281474976710592 & 42 & 281474976710172 \\ \hline
%        50 & 562949936643600 & 20 & 562949919864320 \\ \hline
    \end{tabular}
    \caption{Some terms of sequences}
    \label{tab: terms}
\end{table}

\begin{comment}
\begin{table}[]
    \centering
    \begin{tabular}{|c|c|} \hline
        degree &  pair $(m,n)$ such that $\Delta_{n,m}$ has the given degree \\ \hline
        1 &  $(1,1), (1,2), (2,4), (3,3)$\\
        2 &  $(1,3), (1,4), (1,6), (2,6), (2,8), (2,12)$\\
        3 &  $(3,6), (4,4)$\\
        4 &  $(1,5), (1,8), (1,10), (1,12), (2,10), (2,16), (2,20), (2,24)$\\
        5 &  $\varnothing$\\
        6 &  $(1,7),(1,9),(1,14), (1,18), (2,14), (2,18), (2,28), (2,36), (3,9), (3,12), (3,18), (4,8)$\\
        7 &  $\varnothing$\\ \hline
    \end{tabular}
    \caption{possible pairs $(m,n)$}
    \label{tab: mn_pairs}
\end{table}
\end{comment}

\begin{table}[]
    \centering
    \begin{tabular}{|c|c|c|} \hline
         degree & $(m,n)$    & $\Delta_{n,m}(C/4)$ for $f_c(z) = z^2 + c$ \\\hline
         1      & $(1,1)$    & $C - 1$ \\
                & $(1,2)$    & $C + 3$ \\
                & $(2,4)$    & $-C - 5$ \\
                & $(3,3)$    & $C + 7$ \\ \hline
         2      & $(1,3)$    & $C^{2} + C + 7$\\
                & $(1,4)$    & $C^{2} - 2C + 5$\\
                & $(1,6)$    & $C^{2} - 3C + 3$\\
                & $(2,6)$    & $C^{2} + 9C + 21$\\
                & $(2,8)$    & $C^{2} + 8C + 17$\\
                & $(2,12)$   & $C^{2} + 7C + 13$\\\hline
        3       & $(3,6)$    & $C^{3} + 8C^{2} + 18C + 81$\\
                & $(4,4)$    & $C^{3} + 9C^{2} + 27C + 135$\\\hline
        4       & $(1,5)$    & $C^{4} + C^{3} + C^{2} - 9C + 31$ \\
                & $(1,8)$    & $C^{4} + 2C^{2} - 16C + 17$ \\
                & $(1,10)$   & $C^{4} - 3C^{3} + 9C^{2} - 17C + 11$ \\
                & $(1,12)$   & $C^{4} + 2C^{3} - C^{2} - 14C + 13$ \\
                & $(2,10)$   & $C^{4} + 17C^{3} + 108C^{2} + 313C + 341$ \\
                & $(2,16)$   & $C^{4} + 16C^{3} + 96C^{2} + 256C + 257$ \\
                & $(2,20)$   & $C^{4} + 15C^{3} + 85C^{2} + 215C + 205$ \\
                & $(2,24)$   & $C^{4} + 16C^{3} + 95C^{2} + 248C + 241$ \\ \hline
        5       & $\varnothing$ & $\varnothing$ \\\hline
        6       & $(1,7)$ & $C^6 + C^5 + C^4 + C^3 + 15C^2 - 97C + 127$\\
                & $(1,9)$ & $C^6 - 6C^4 + 7C^3 + 36C^2 - 102C + 73$ \\
                & $(1,14)$ & $C^6 - 3C^5 + 9C^4 - 27C^3 + 67C^2 - 89C + 43$ \\
                & $(1,18)$ & $C^6 + 6C^4 - 9C^3 + 36C^2 - 90C + 57$ \\
                & $(2,14)$ & $C^6 + 25C^5 + 261C^4 + 1457C^3 + 4589C^2 + 7737C + 5461$ \\
                & $(2,18)$ & $C^6 + 24C^5 + 240C^4 + 1281C^3 + 3852C^2 + 6192C + 4161$ \\
                & $(2,28)$ & $C^6 + 23C^5 + 221C^4 + 1135C^3 + 3285C^2 + 5079C + 3277$ \\
                & $(2,36)$ & $C^6 + 24C^5 + 240C^4 + 1279C^3 + 3828C^2 + 6096C + 4033$ \\
                & $(3,9)$ & $C^6 + 16C^5 + 98C^4 + 415C^3 + 1436C^2 + 2482C + 5329$ \\
                & $(3,12)$ & $C^6 + 16C^5 + 96C^4 + 382C^3 + 1268C^2 + 2080C + 4225$ \\
                & $(3,18)$ & $C^6 + 16C^5 + 94C^4 + 351C^3 + 1116C^2 + 1710C + 3249$ \\
                & $(4,8)$ & $-C^6 - 12C^5 - 47C^4 - 188C^3 - 527C^2 - 4913$ \\ \hline
        7       & $\varnothing$ & $\varnothing$ \\ \hline
    \end{tabular}
    \caption{$\Delta_{n,m}(c)$ with given degree}
    \label{tab: Delta_mn}
\end{table}

% --------------------------------------------------------------------------

\newpage
\section{Exact description of Newton polygons for $z^{d+1} + cz$ and $z^{d+1} - cz + c$} \label{sec: Newton_polygon}

% --------------------------------------------------------------------------

For a positive integer $ n $ and $ f(z) = z^{d+1} + cz $ or $ f(z) = z^{d+1} - cz + c $, we show Newton polygons of $ f^{\circ n} (z) $ in \cref{fig: Newton_polygon_z_to_d+1_+cz,fig: Newton_polygon_z_to_d+1_c_z_to_d_+c}.
The proof is by induction on $n$.
\begin{figure}[h]
	\begin{minipage}[bthp]{0.45\linewidth}
		\centering
		\includegraphics[width = 7.5cm]{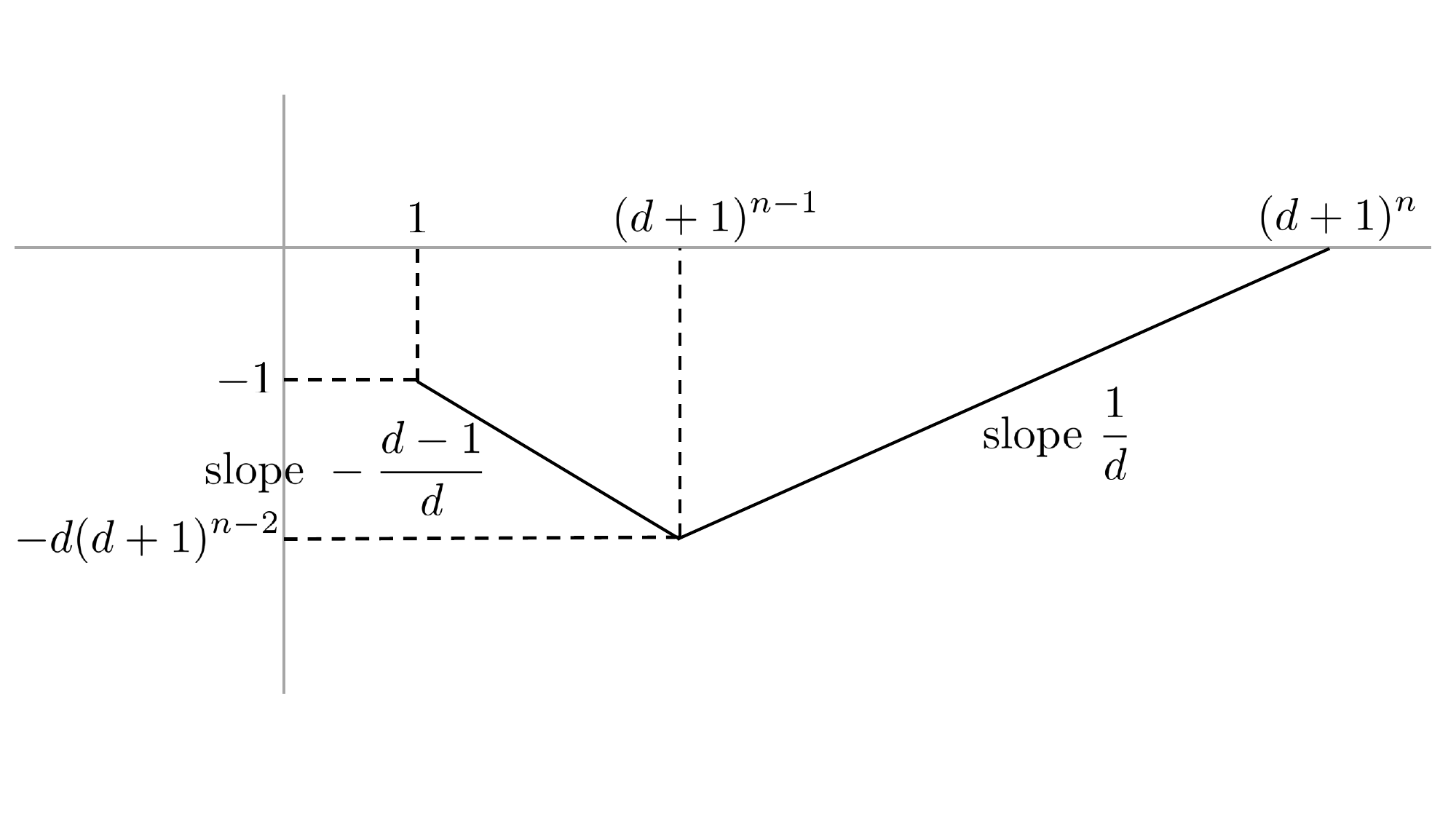}
		\caption{The Newton polygon of $ f^{\circ n}(z) $ for $ f(z) = z^{d+1} + cz $.}
		\label{fig: Newton_polygon_z_to_d+1_+cz}
	\end{minipage}		
	\begin{minipage}[bthp]{0.45\linewidth}
		\centering
		\includegraphics[clip,width = 7.5cm]{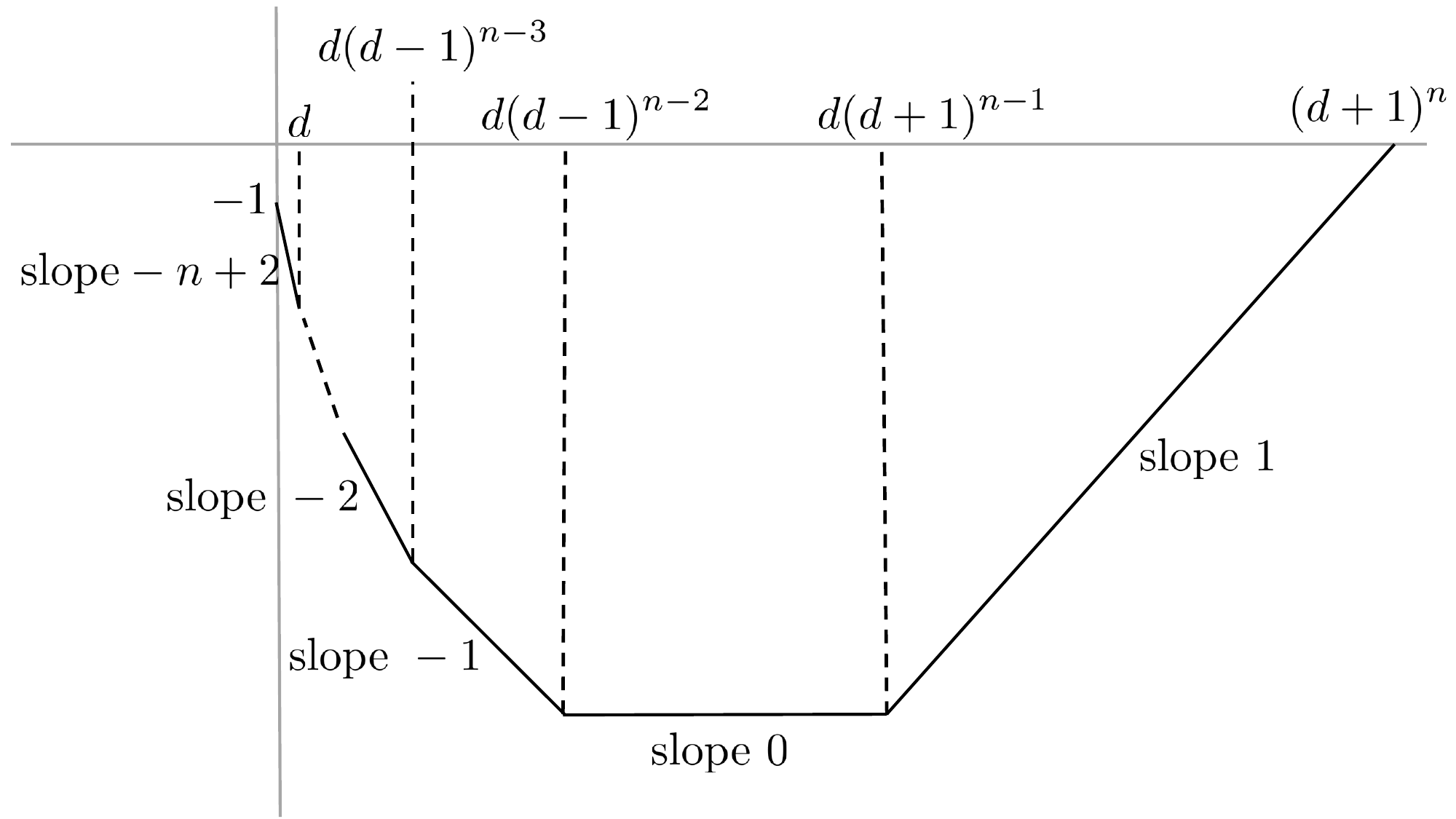}
		\caption{The Newton polygon of $ f^{\circ n}(z) $ for $ f(z) = z^{d+1} - cz + c $.}
		\label{fig: Newton_polygon_z_to_d+1_c_z_to_d_+c}
	\end{minipage}
\end{figure}

% --------------------------------------------------------------------------
%		References
% --------------------------------------------------------------------------
\newpage
\bibliographystyle{alpha}
\bibliography{dynamic}

% --------------------------------------------------------------------------
\end{document}